\documentclass{article}
\usepackage[utf8]{inputenc}

\usepackage{amsthm,amsmath,stmaryrd,bbm,hyperref,geometry,color,authblk}
\usepackage{amssymb}
\usepackage[english]{babel}
\usepackage{graphicx}
\usepackage{amsfonts,amssymb}
\usepackage{verbatim}
\usepackage{enumerate}
 
\newcommand{\po}{\left(}
\newcommand{\pf}{\right)}
\newcommand{\co}{\left[}
\newcommand{\cf}{\right]}
\newcommand{\cco}{\llbracket}
\newcommand{\ccf}{\rrbracket}
\newcommand{\R}{\mathbb R}

\newcommand{\N}{\mathbb N} 
\renewcommand{\H}{\mathbb H} 
 
\newcommand{\tF}{\widetilde{\mathcal F}} 
\newcommand{\dd}{\text{d}}
\newcommand{\bX}{\mathbf{X}}
\newcommand{\bx}{\mathbf{x}}
\newcommand{\by}{\mathbf{y}}
\newcommand{\bY}{\mathbf{Y}}
\newcommand{\na}{\nabla}

\newtheorem{theorem}{Theorem}
\newtheorem{assumption}{Assumption}
\newtheorem{lemma}[theorem]{Lemma}

\newtheorem{proposition}[theorem]{Proposition}
\newtheorem{remark}{Remark}

\title{A two-scale approach for sampling local free energy minimisers}

\author[1]{Antoine Leclerc}
\author[2,3]{Pierre Monmarché\thanks{pierre.monmarche@univ-eiffel.fr}}
\affil[1]{Mines de Saint-Etienne, France}
\affil[2]{LAMA, Université Gustave Eiffel, France}
\affil[3]{Institut Universitaire de France}
\date{\today}

\begin{document}

\maketitle

\begin{abstract}
When a free energy admits several local minimisers, the associated mean-field particle system exhibits a metastable behavior, undergoing random transitions between these minimisers. However, when the number of particles is large, these transitions are rare events and typically occur at a time-scale which is out of reach for simulations. In this work, we introduce a system of particles coupled with a slow macroscopic external variable which approximately follows a Langevin dynamics associated to a coarse-grained free energy. The empirical distribution of the particle system thus samples all local minimisers of the initial free energy, with transition times which are now independent from the number of particles. The construction is inspired by a theoretical tool: the two-scale approach for proving log-Sobolev inequalities for Gibbs measures.
\end{abstract}

\section{Introduction}

\subsection{Overview}

\paragraph{Motivation.} Consider a so-called free energy 
\begin{equation}
\label{eq:energieLibre}
\mathcal F(\rho) = \mathcal E(\rho) + \sigma^2 \mathcal H(\rho)\,,
\end{equation}
defined over $\mathcal P_2(\R^d)$ the set of probability measures $\rho$ with finite second moment, where $\mathcal E:\mathcal P_2(\R^d) \rightarrow (-\infty,\infty]$ is some energy functional, $\mathcal H(\rho) = \int_{\R^d} \rho \ln \rho$ is the entropy (equal to $+\infty$ if $\rho$ doesn't have a Lebesgue density) and $\sigma^2>0$ is a temperature parameter, which we fix equal to $1$ without loss of generality from now on (unless otherwise specified, when discussing specific models, particularly for  numerical experiments). We are interested in the situation where $\mathcal F$ has several local minimisers, and our goal is to sample them. One of the main interest is to find, among these local minimisers, the global ones, but finding all local minimisers is also of interest. Indeed, these  are stable solutions of the gradient flow associated to $\mathcal F$, and their study is related for instance to questions of phase transitions, clustering, localization in statistical physics~\cite{tugaut2023steady,bashiri2020gradient,sandier2004gamma,HairerPavliotis}. The question of minimizing free energies over $\mathcal P_2(\R^d)$ also arises in machine learning~\cite{Szpruch,mei2018mean,geshkovski2025mathematical}, optimisation~\cite{chizat,peyre2015entropic}, mathematical biology~\cite{bunne2022proximal}, variational inference~\cite{arbel2019maximum,lambert2022variational} and sampling~\cite{lelievre2026convergence}. For instance, clustering is an important feature in the attention mechanism used in transformers~\cite{geshkovski2025mathematical,ShalovaSchlichting+2026}.

The gradient flow of $\mathcal F$ (with $\sigma^2=1$) is the solution of 
\begin{equation}
\label{eq:gradientflowF}
\partial_t \rho_t = \na\cdot \po \rho_t D\mathcal E \pf + \Delta \rho_t
\end{equation}
with the intrinsic derivative $D\mathcal E$ given in~\eqref{def:intrinsic} below. Similarly to the finite-dimensional situation, a way to sample different local minimiser is to follow~\eqref{eq:gradientflowF} starting from different random initial conditions (for instance normal distributions with random means and covariance matrices). However, in order to find several minimisers of a function $f\in\mathcal C^1(\R^d,\R)$, an alternative to randomized initialisations is to perturbate the gradient descent with a Brownian noise, leading to the so-called overdamped Langevin diffusion process, solving
\begin{equation}
\label{eq:Langevin}
\dd X_t = -\na f(X_t) \dd t + \sqrt{2/\beta} \dd B_t\,,
\end{equation}
with $\beta^{-1}$ a temperature parameter and $B$ a $d$-dimensional Brownian motion. Under suitable mild conditions on $f$, this process is ergodic with respect to the Gibbs measure with density proportional to $\exp(-\beta f)$. In particular, given enough time, it will pass near all local minimisers. When $\beta$ is small, the process behaves like a Brownian motion, and thus in high dimension it will take a very long time to actually get near a given local minimiser.  When $\beta$ is large, the process is metastable, i.e. it spends most of its time in the vicinity of local minimisers, jumping from one to the other along some rare transitions. A situation where~\eqref{eq:Langevin} is particularly relevant is when the set of local minimisers of interest are all in the vicinity of a manifold of dimension much smaller than the ambient space $\R^d$. This is for instance the case in molecular simulations, where $x$ is the position of nuclei and $f$ is the energy of the system. In this situation,  $d=3n$ with $n$ the number of atoms (ranging from $100$ to $10^9$ in typical applications), but the configurations of atoms are very structured (for instance the interatomic bonds in a protein are fixed and prevent high variations of the corresponding interatomic distances) so that the region of low energies, containing all stable configurations of interest, is concentrated around a manifold of dimension much smaller. However, no parametrization of this manifold is available. In particular, the idea of following the  gradient flow of $f$ starting from random initial conditions is not practicable, as no good initial distribution is known; when resorting to uninformed distributions (e.g. isotropic Gaussian), with overwhelming probability, the gradient flow get stuck in a completely non-physical configuration (i.e. a local minimiser $f$ with an extremely large value with respect to $\inf f$). However, provided that a single reasonable $x_0$ in the low-energy region is already known, we can use~\eqref{eq:Langevin} to explore locally the manifold, taking $\beta$ large enough to stay in its vicinity and small enough to undergo the energy barriers within the manifold. From this simple strategy, multi-scale variations can be considered, for instance by using a trajectory at high temperature $\beta^{-1}$ to give initial conditions for trajectories at smaller temperature, by running a single trajectory with a temperature schedule $t\mapsto \beta_t^{-1}$ that decays with time (simulated annealing  or sequential Monte Carlo~\cite{del2006sequential,annealing,chehab2024provable}) or by having several trajectories exchanging their temperatures at random times (parallel tempering or replica exchange~\cite{Marinari_1992,neal1996sampling}).

\medskip

In view of this discussion, our goal is to provide an analogue of~\eqref{eq:Langevin} when $(f,\R^d)$ is replaced by $(\mathcal F,\mathcal P_2(\R^d))$. This requires to perturbate the gradient flow~\eqref{eq:gradientflowF} with a stochastic noise. 

\paragraph{Related literature.} This question has already been tackled with various approaches: using the theory of Dirichlet forms on metric spaces \cite{dello2022dirichlet,overbeck1995analytic,ren2024diffusion,von2009entropic,Shao,ren2024markov}, working  in a space of functions in bijection with (a subset of) $\mathcal P(\R^d)$, for instance repartition functions in one dimension \cite{delarue2024rearranged,delarue2024rearranged2,von2009entropic} or densities  in some Hilbert space with a trace-class Gaussian distribution as the reference measure  \cite{EberleBouRabee}, or perturbating the infinite dimensional equation~\eqref{eq:gradientflowF} by a finite-dimensional noise~\cite{delarue2024ergodicity,germain2025stochastic}. There are various motivations for these stochastic perturbations: the goal might be to define a canonical reversible process with respect to the entropy (which would thus really be an infinite-dimensional analogue of the Langevin dynamics~\eqref{eq:Langevin}), or to exploit  the regularization properties of the noise to get well-posed equations, as in \cite{delarue2024rearranged}, or to  induce ergodicity (i.e. uniqueness of a stationary distribution) in cases where the deterministic flow has several stationary solutions, as in \cite{ANGELI2023127301,delarue2024ergodicity,germain2025stochastic,maillet2023note,Germain_applied}.

As we will see, by some aspects, the viewpoint taken in the present work is close to the one in \cite{delarue2024ergodicity,germain2025stochastic} since we will only need a finite-dimensional Brownian noise (which is convenient for practical implementation) and, as motivated above, our primary goal is to ensure the ergodicity of the dynamics, in order to visit all local minimisers of the free energy. However, in~\cite{delarue2024ergodicity,germain2025stochastic}, the fact that the initial equation is a gradient flow doesn't play any role, and the measure which is sampled by the stochastic process is completely unknown. Noise is simply there to allow for exploration. If the noise is small enough then the stochastic dynamics can be expected to be attracted by stable solutions of the deterministic equation when it gets close to it, but the behavior of the process outside the vicinity of these stable solutions  is not controlled (hence, neither is the time spent on average far from these stable solutions). By contrast, our stochastic perturbation will be designed to target a specific equilibrium measure, highly relying on the gradient flow structure of~\eqref{eq:gradientflowF} and the specific form of the free energy $\mathcal F$. For this reason, it is less versatile than the approach in~\cite{delarue2024ergodicity,germain2025stochastic} but  closer   to the Langevin process~\eqref{eq:Langevin}. In particular, the time spent by the process in the vicinity of local minimisers and the transition times between different minimisers is controlled as in the Langevin case~\eqref{eq:Langevin} (in terms of $\beta$).

\paragraph{Organisation.} The rest of this work is organised as follows. In the rest of the introduction, we introduce our two-scale sampler (first, by presenting the mean-field particle system associated to~\eqref{eq:gradientflowF} in Section~\ref{subsec:introparticles}, then describing the two-scale approach and motivating the sampler in a simple case with quadratic interaction in Section~\ref{sec:intro_twoscale}, finally extending this definition to more general cases in Section~\ref{subsec:generalcase}), and compare it to the family of enhanced-sampling methods based on collective variables in Section~\ref{sec:CVenhancesampling}. Some general notations are gathered in Section~\ref{subsec:def-not}. Two approximations are made in the definition of the process, since a finite number $N$ of particles  and a positive time-scale separation parameter $\varepsilon$ are used: the theoretical results describing the convergence of the process to the idealized sampler as $N\rightarrow \infty$ and $\varepsilon\rightarrow 0$ are stated in Section~\ref{sec:theory}. Section~\ref{sec:numerique} is devoted to numerical experiments. Finally, the proofs of the theoretical results are provided in Section~\ref{sec:proof}.

\subsection{Mean-field interacting particle system}\label{subsec:introparticles}

The system of $N$ interacting particles associated to the free energy $\mathcal F$ is $\bX_t = (X_t^1,\dots,X_t^N)$ solving
\begin{equation}
\label{eq:LangevinN}
\dd \bX_t = -\na U_N(\bX_t) \dd t  + \sqrt{2}\dd \mathbf{B}_t,
\end{equation}
with $\mathbf{B}$ a $dN$-Brownian motion and
\[U_N(\bx) = N \mathcal E\po \pi_{\bx}\pf,\qquad \pi_{\bx} = \frac1N\sum_{i=1}^N \delta_{x_i}\,.\]
Under suitable conditions on $\mathcal E$, it is well-known that for all $t\geqslant 0$, $\pi_{\bX_t} $ converges almost surely to the solution $\rho_t$ of~\eqref{eq:gradientflowF}, provided this holds at $t=0$. This is referred to as the mean-field limit (associated to the propagation of chaos phenomenon, according to which $X_t^1$ and $X_t^2$ are asymptotically independent as $N\rightarrow \infty$ provided this holds at time $t=0$). The fluctuations of this convergence have been studied in \cite{dawson1983critical,fernandez1997hilbertian} (see~\cite{bernou2026uniform} for further more recent works) and can be informally described by the fact that, in some weak sense,
\[\dd \pi_{\bX_t} \underset{N\rightarrow\infty} \simeq -\na_{\mathcal W_2}\mathcal F(\pi_{\bX_t}) \dd t + \frac{1}{\sqrt{N}} \sqrt{\pi_{\bX_t}} \dd \xi_t\]
for some Gaussian noise $\xi$, where $-\na_{\mathcal W_2}\mathcal F$ stands for the right hand side of~\eqref{eq:gradientflowF}. At this point, we may think that we have achieved our goal since the empirical distribution of particles evolves according to a stochastic perturbation of the gradient flow. However, we need $N$ to be large to have a good approximation of the gradient flow, and in that case the noise intensity is very small. Transitions between different local minimisers occur at a time-scale of order $e^{cN}$ for some $c>0$ \cite{Monmarchemetastable}, and thus are never seen in practical simulations.

By contrast, we are aiming at defining a process where the noise persists in the limit $N\rightarrow \infty$.

\medskip

At a fixed $N$, $(\bX_t)_{t\geqslant 0}$ is an overdamped Langevin diffusion and, under suitable conditions, it is ergodic with respect to the Gibbs measure
\[\mu_N \propto \exp(-U_N)\,.\]
Under this measure, $\pi_{\bX}$ is concentrated around the minimisers of $\mathcal F$ for large $N$. However, when $\mathcal F$ has several local minimisers, sampling this measure with~\eqref{eq:LangevinN} requires a time which is exponentially large with $N$ since we have to wait for $\pi_{\bX_t}$ to undergo several rare transitions, and is thus unfeasible in practice. This can be interpreted in terms of the log-Sobolev constant of $\mu_N$, which is the largest constant $\lambda_N>0$ such that
\begin{equation}
\label{eq:LSI}
\forall \nu\in \mathcal P_2(\R^{dN}),\qquad \mathcal H\po \nu|\mu_N\pf \leqslant \frac{1}{2\lambda_N} \mathcal I(\nu|\mu_N )\,
\end{equation}
(with the relative entropy and Fisher information given in~\eqref{eq:HI}). The long-time convergence rate of~\eqref{eq:LangevinN} to $\mu_N$ in relative entropy is $\lambda_N$, see \cite[Theorem 5.2.1]{BakryGentilLedoux}. When $\mathcal F$ has several local minimisers, it vanishes exponentially fast with $N$ \cite{Monmarchemetastable} (see \cite{lelievre2012two} for further details on this topic in finite dimension).

\subsection{Two-scale approach}\label{sec:intro_twoscale}

Our work is motivated by the two-scale approach used in \cite{bauerschmidt2019very,bauerschmidt2025criterion,Monmarchemetastable,MonmarcheEquivalence} (see references within for previous works) for establishing log-Sobolev inequalities (LSI)~\eqref{eq:LSI} for $\mu_N$ (with $\lambda_N$ independent from $N$ in these works). It relies on a representation of $\mu_N$ in the spirit of De Finetti's theorem, i.e. as a mixture of tensorised measures (or at least simpler Gibbs measures for which a uniform-in-$N$ LSI is known by some convexity argument \cite{wang2024uniform,M61}).  Let us describe this in the simple case where
\begin{equation}
\label{eq:ECurieWeiss}
\mathcal E(\rho) = \int_{\R^d} V(x) \rho(x)\dd x + \frac14 \int_{\R^d} |x-y|^2 \rho(x)\rho(y) \dd x \dd y = \int_{\R^d} V_1(x)\rho(x)\dd x - \frac12 \po \int_{\R^d} x\rho(x)\dd  x\pf^2
\end{equation}
for some $V\in\mathcal C^1(\R^d,\R)$ and $V_1(x) = V(x) + \frac12|x|^2$. Then,
\[U_N(\bx) = \sum_{i=1}^N V_1(x_i) - \frac1{2N} \po \sum_{i=1}^N x_i\pf^2\,. \]
 Relying on the formula for the moment generating function of Gaussian distribution,
\[\mathbb E \po \exp\po z \cdot G\pf  \pf = \exp\po \frac\kappa2 |z|^2 \pf\,,\qquad z\in\R^d,\ G\sim \mathcal N\po 0,\kappa^{-1}I_d\pf ,  \]
applied with $z=\sum_{i=1}^N x_i$ and $\kappa=1/N$, we can write
\[\int_{\R^{dN}} \varphi(\bx)e^{-U_N(\bx)}\dd \bx = \mathbb E \po  \int_{\R^{dN}} \varphi(\bx) e^{ - \sum_{i=1}^N \co V_1(x_i) - G\cdot x_i\cf} \dd \bx  \pf  \]
for any $\varphi\in L^\infty(\R^{dN})$. This amounts to the representation
\begin{equation}
\label{eq:muNmixtureQuadra}
\mu_N = \int_{\R^d} \mu_\theta^{\otimes N} \nu_N(\theta)\dd \theta
\end{equation}
with 
\[\mu_\theta(x) = \frac{1}{Z(\theta)} e^{-V_1(x) + \theta\cdot x}\,,\qquad Z(\theta) = \int_{\R^d} e^{-V_1(x) + \theta\cdot x} \dd x\,, \]
and
\[\nu_N(\theta) \propto \exp \po - N \omega(\theta)\pf,\qquad \omega(\theta) = \frac12 |\theta|^2 - \ln Z(\theta)\,.\]
In other words, to sample $\mu_N$, we can sample $\theta \sim \nu_N$ and then generate i.i.d. samples $X^1,\dots,X^N$ with law $\mu_\theta$, which can be done with
\begin{equation}
\label{eq:particulesiidtheta}
\dd X_t^i = - \na V_1(X_t^i) \dd t + \theta \dd t + \sqrt{2}\dd B_t^i\,.
\end{equation}
Since these particles are independent, for a fixed $\theta$, the long-time convergence of these toward $\mu_\theta^{\otimes N}$ is independent from $N$ and given by the log-Sobolev constant of $\mu_\theta$. Assuming for instance that $V_1$ is the sum of a strongly convex and a bounded functions, it follows by standard arguments that this LSI constant can be bounded from below uniformly over $\theta$ (see~\cite{BakryGentilLedoux} or the proof of Proposition~\ref{prop:lambda_thetaPL} below for details).  In other words, sampling $\mu_{\theta}^{\otimes N}$ can be done in a time independent from $\theta$ and $N$, contrary to $\mu_N$.

All the sampling difficulty is now encoded in $\theta$,  the critical points of $\omega$ corresponding to those of $\mathcal F$. Indeed, on the one hand,
\[\na\omega(\theta) =  \theta - \int_{\R^d} x \mu_\theta(x) \dd x\,. \]
On the other hand, by classical considerations (see e.g.~\cite{tugaut2023steady}), for the energy~\eqref{eq:ECurieWeiss}, a density $\rho_*$ is a stationary solution of~\eqref{eq:gradientflowF} if and only if
\begin{equation}
\label{eq:self-consistent}
\rho_*(x) \propto \po -V_1(x) +  m_*\cdot x  \pf\,,\qquad m_* = \int_{\R} x \rho_*(x)\dd x\,, 
\end{equation}
namely $\rho_* = \mu_{m_*}$. Multiplying this equation by $x$ and integrating shows that $m_* = \int_{\R^d} x\mu_{m_*}(x)\dd x$, which means that $m_*$ is a critical point of $\omega$. Conversely, if $m_*$ is a critical point of $\omega$, it is straightforward to check that $\mu_{m_*}$ satisfies the self-consistency equation~\eqref{eq:self-consistent}, hence is a critical point of $\mathcal F$.

We can go a step further and in fact see that $\theta_*$ is a local minimiser of $\omega$ if and only if $\mu_{\theta_*}$ is a local minimiser of $\mathcal F$, cf. \cite[Section 3.1.2]{monmarche2025local}. In other words, in the situation where $\mathcal F$ has several local minimisers, this will be the case for $\omega$ and then trying to sample  $\nu_N$ with
\begin{equation}
    \label{eq:ThetaLangevinintro}
    \dd \Theta_t = -\na \omega(\Theta_t)\dd t + \sqrt{2/N}\dd B_t
\end{equation}
would be unpracticable  for large $N$ due to metastability \cite{lelievre2012two}.

As a conclusion of this discussion, we can say that the question of sampling the local minimisers of $\mathcal F$ boils down to sampling the Gibbs measure $\nu_N$  associated to the coarse-grained/macroscopic potential $\omega$, and the difficulty is that the temperature $1/N$ is vanishing as $N$ increases.
 
 Our approach is then simple: it consists in sampling $\theta$ according to $\nu_\beta \propto \exp \po - \beta \omega\pf$ with $\beta>0$ fixed independent from $N$. We postpone to Section~\ref{sec:CVenhancesampling} a discussion about similar ideas in different contexts.
 
The remaining problem is that $\na \omega$ cannot be computed in practice. Hence, we resort to a stochastic gradient approximation. Indeed, we have seen that sampling $\mu_\theta$ with i.i.d. particles~\eqref{eq:particulesiidtheta} is easy (in the sense that the time needed is independent from $N$ and $\theta$; due to $V_1$ possibly being non-convex, this can still be long in some cases), and then we can estimate
\begin{equation}
\label{loc:approx}
\na \omega(\theta) \simeq \theta - \frac1N\sum_{i=1}^N X_t^i
\end{equation}
In fact, we sample both the macroscopic parameter $\theta_t$ and the particles simultaneously, while adding a small speed parameter $\varepsilon>0$ to the dynamics of $\theta_t$. Indeed, if $\theta_t$ moves slowly enough, we expect the particles to reach equilibrium as if $\theta_t$ were constant, so that the approximation~\eqref{loc:approx} remains valid (an alternative would have been to have the particles influenced by an exponential moving time-average of $\Theta_t$, in the spirit of~\cite{Yulong}).

As a conclusion, in the case of~\eqref{eq:ECurieWeiss} with a quadratic interaction energy, the process that we consider in the present work is given by
\begin{align*}
\forall i\in\cco 1,N\ccf,\qquad \dd X_t^i &= - \na V_1(X_t^i) \dd t + \Theta_t \dd t + \sqrt{2}\dd B_t^i\\
\dd \Theta_t &= -\varepsilon \po \Theta_t - \frac1N\sum_{i=1}^N X_t^i\pf \dd t + \sqrt{2\varepsilon/\beta} \dd B_t\,.
\end{align*}

\subsection{More general case}\label{subsec:generalcase}

The previous discussion can be extended to more general situations. For instance, consider the energy
\begin{equation}
\label{eq:ECurieWeiss-extended}
\mathcal E(\rho) = \int_{\R^d} V(x) \rho(x)\dd x + \frac12 \int_{\R^d} W(x,y) \rho(x)\rho(y) \dd x \dd y 
\end{equation}
with $V\in\mathcal C^1(\R^d)$, $W\in\mathcal C^1(\R^{2d})$ such that $W(x,y)=W(y,x)$ for all $x,y\in\R^d$. Under mild conditions, we can write a Mercer decomposition of $W$ of the form
\begin{equation}
\label{eq:Mercer}
W(x,y) = W_0(x) + W_0(y) + \sum_{k\geqslant 0} m_k(x)m_k(y)  - \sum_{k\geqslant 0} n_k(x)n_k(y)\,,
\end{equation}
for some functions $W_0$, $(n_k,m_k)_{k\in\N}$. For instance, for $\alpha>0$,
\begin{equation}
\label{eq:Gaussianinteraction}
 - \alpha e^{-\frac{(x-y)^2}2} = -\sum_{k\geqslant 0} \alpha e^{-\frac{x^2}{2}} \frac{x^ky^k}{k!} e^{-\frac{y^2}{2}}\,. 
\end{equation}
From the representation~\eqref{eq:Mercer}, the energy can be written
\[\mathcal E(\rho) = \int_{\R^d} V_1(x)\rho(x) \dd x + \frac12\sum_{k\geqslant 0}  \po \int_{\R} m_k(x)\rho(x)\dd x\pf^2  - \frac12 \sum_{k\geqslant 0} \po \int_{\R} n_k(x)\rho(x)\dd x\pf^2\,,\]
with $V_1 = V + W_0$. The two first terms are convex as a function of $\rho$ (along flat interpolations $t\mapsto (1-t)\rho_0 + t \rho_1$), and thus they will not be the problematic ones. Denoting $\varphi(x) = (n_k(x))_{k\geqslant 0}$, we get an expression of the form
\[\mathcal E(\rho) = \mathcal E_c(\rho) - \frac12\|\rho(\varphi)\|_{\ell^2}^2\,,\]
with $\rho \mapsto \mathcal E_c(\rho)$ which is convex. The  dependency of the last part as a function of $\rho$ is quadratic since the energy~\eqref{eq:ECurieWeiss} only involves pairwise interactions. A natural extension to more general interactions is given by energies of the form 
\[\mathcal E(\rho) = \mathcal E_c(\rho) + R \po \rho(\varphi)\pf \,,\]
with $\varphi$ taking value in some Hilbert space $\mathbb H$ and $R:\mathbb H \rightarrow \R$. However, in this situation, a reasonable assumption is that the Hessian of $R$ is lower bounded, meaning that $\psi \mapsto R(\psi) + L \|\psi\|^2$ is convex for some large $L$, and in that case we can write
\[ \mathcal E_c(\rho) + R \po \rho(\varphi)\pf = \widetilde{\mathcal E}_c(\rho) - \frac12 \| \rho(\tilde \varphi)\|^2 \,,\]
with a convex $\widetilde{\mathcal E}_c$ and $\tilde \varphi(x)= \sqrt{L}\varphi$, and we are back to the previous form.

As a conclusion, the general form of energy that we wish to cover is
\begin{equation}
\label{eq:geneE}
\mathcal E(\rho) = \mathcal E_0(\rho) - \frac{1}{2}\|\rho(\varphi)\|^2
\end{equation}
for some function $\varphi:\R^d\rightarrow \mathbb H$ with $\mathbb H$ some Hilbert space, and $\mathcal E_0$ is some energy that we can think as convex or at least not too concave to cause difficulties, as we will state precisely later on (specifically, see~\eqref{eq:convexity-c}).

In that case, the associated mean-field potential is 
\[U_N(\bx) = N \mathcal E(\pi_{\bx}) = V_N(\bx) - \frac{1}{2N} \po \sum_{i=1}^N \varphi(x_i)\pf^2\]
with $V_N(\bx) = N \mathcal E_0(\pi_{\bx})$. Reasoning as in the quadratic interaction case (possibly with some subtleties if the dimension of $\mathbb H$ is infinite, see \cite{MonmarcheEquivalence}, but we won't consider this case afterwards), we get for the Gibbs measure $\mu_N \propto \exp(-U_N)$ the representation
\[\mu_N = \int_{\mathbb H} \mu_{\theta}^N \nu_N(\theta) \dd \theta\,,\]
with
\[\mu_\theta^N = \frac{1}{Z_N(\theta)} e^{-V_N(\bx) + \theta\cdot \sum_{i=1}^N \varphi(x_i) }\,,\qquad Z_N(\theta) = \ln \int_{\R^{dN}} e^{-V_N(\bx) + \theta\cdot \sum_{i=1}^N \varphi(x_i) } \dd \bx  \]
(which in general is not a tensor product if $\mathcal E_0$ is not linear) and
\[\nu_N \propto \exp \po - N \omega_N(\theta)\pf \]
with
\begin{equation}
    \label{eq:omegaN}
     \omega_N(\theta) = \frac{1}{2}|\theta|^2 - \frac{1}{N}\ln Z_N(\theta)\,.
\end{equation}
From
\begin{equation}
    \label{eq:naomegaN}
\na \omega_N(\theta) = \theta - \int_{\R^{dN}} \pi_{\bx}(\varphi) \mu_\theta^N(\bx)\dd \bx\,,
\end{equation}
proceeding as in the quadratic case, we end up with the coupled system
\begin{equation}
\label{eq:thetaXcouples}
\left\{
\begin{array}{rcl}
\forall i\in\cco 1,N\ccf,\qquad \dd X_t^i &=& - \na_{x_i} V_N(\bX_t) \dd t + \Theta_t \cdot \na \varphi(X_t^i) \dd t + \sqrt{2}\dd B_t^i\\
\dd \Theta_t &=& - \varepsilon\po \Theta_t - \frac1N\sum_{i=1}^N \varphi(X_t^i)\pf \dd t + \sqrt{2\varepsilon/\beta} \dd B_t\,.
\end{array}\right.
\end{equation}
The introduction of this process, which can be referred to as a two-scale sampler, is our main contribution. Its theoretical and empirical study is the main topic of the rest of this work.

In practice we only consider the finite dimensional case $\H = \R^m$ for some $m\geqslant 1$. From an infinite-dimensional situation, we have to consider an orthogonal decomposition $\varphi=(\varphi_1,\varphi_2) \in \mathbb H_1 \oplus \H_2$ with $\H_1$ of finite dimension and then the process is~\eqref{eq:thetaXcouples} but with $\varphi$ replaced by $\varphi_1$ and $V_N $ replaced by $N \mathcal E_1(\pi_{\bx})$ with $\mathcal E_1(\rho) = \mathcal E_0(\rho) - \frac12\|\rho(\varphi_2)\|^2$. In that case, $\mathcal E_1$ might not be convex but, provided $\|\varphi_2\|_\infty$ is small enough for instance, it will still have good properties for our purpose (see~\cite{M61} or Assumption~\ref{assum:PL} below).

\subsection{Relation to collective variable-based enhanced sampling methods}
\label{sec:CVenhancesampling}

Metastability is a central difficulty for sampling in molecular simulations~\cite{lelievre2010free}. On the other hand, in this field, the problem is typically multi-scale, with fast oscillations of the nuclei positions $x\in\R^d$ and slow motion of macroscopic configurations (described by much less degrees of freedom than $d$). The context is thus very similar to the one we are facing in our mean-field settings.  In this situation, over the last twenty years, a broad family of adaptive biasing methods have been developed for molecular simulations, based on collective variables. A collective variable is described by a function  $\xi:\R^d \rightarrow \R^m$ with $m\ll d$, often $m\leqslant 3$. Ideally, it should be chosen so that the value of $\xi(x)$ describes which macroscopic configuration the microscopic configuration $x$ belongs to (i.e. $\xi(x)$ is an order parameter). In other words, denoting by $U$ the energy of the system (so that we aim at sampling the Gibbs measure with density $\mu_\beta \propto \exp(-\beta U)$ with $\beta$ the inverse temperature), the value of $\xi(x)$ should ideally characterise the local minimiser of $U$ whose basin of attraction (for the gradient flow) contains $x$. Initially, the definition of suitable collective variables were based on expert knowledge. Over recent years, the question of their unsupervised learning  has been a very active research field, see e.g. \cite{belkacemi2022chasing} and references within.

Now, suppose that we are given a collective variable $\xi$. Assuming that $\xi(X)$ has a density $\nu_\beta$ when $X\sim \mu_\beta$, we write $A(\psi) = -\beta^{-1} \ln \nu_\beta(\psi)$. This function is called the free energy associated to $\xi$. If the collective variable is ``good" (whatever this precisely means), then the different modes of $\mu_\beta$ should give different modes of $\nu_\beta$. In other words, the multimodality of $\mu_\beta$ should transfer to $\nu_\beta$, and the local minimisers of $A$ should correspond to stable macroscopic configurations. Moreover, if the different modes of $\mu_\beta$ are separated in $\nu_\beta$, then there should not be any remaining multimodality in the conditional law of $X$ given $\xi(X) = \psi$, whatever $\psi\in\R^m$. In this situation, the main issue is to sample correctly $\xi(X)$ (which is difficult since, relying on a standard Langevin dynamics for $X$, transitions of $\xi(X)$ between modes are rare events due to metastability), while sampling the conditional distributions is easy. This situation is very similar to the one discussed in Section~\ref{sec:intro_twoscale} in our mean-field situation.

In this context, a variety of adaptive importance methods are based on the idea to target a measure $\tilde \mu_\beta$ such that:
\begin{enumerate}
\item The conditional law of $X$ given $\xi(X) = \psi$, for any $\psi \in \R^m$, is the same  under $\mu_\beta$ and $\tilde \mu_\beta$.
\item The marginal distribution of $\xi(X)$ when $X\sim \tilde \mu_\beta$ is $\nu_{\beta_1} \propto \exp(-\beta_1 A)$ for some $\beta_1< \beta$.
\end{enumerate}
At the end of the simulation, expectations with respect to $\mu_\beta$ are then computed by importance reweighting of samples distributed according to $\tilde \mu_\beta$. The motivation of the definition of $\tilde \mu_\beta$ is that increasing the temperature $\beta_1^{-1}$ reduces the transition time between modes. In some applications, $\xi(x)$  is in a compact torus (for instance when it is given by some interatomic angles), and then we may simply take $\beta_1 = 0$ (i.e. $\xi(X)$ is sampled uniformly). Moreover, this tempering is only done on a small-dimensional marginal distribution, instead of simply sampling $\mu_{\beta_1}$, which would be a reference measure very far from the target distribution $\mu_\beta$ due to the high dimension $d$, leading to a large variance of the importance sampling weights and a poor result.

Standard algorithms with wide-spread applications which target $\nu_\beta$ are for instance the adaptive biasing force (ABF) method, metadynamics or OPES, see the discussion and references in \cite[Section 1.2]{lelievre2026convergence}.

\medskip

As we see, the two-scale sampler introduced in the present work follows exactly the same objective as these methods. We aim at sampling $(\bX,\Theta) $ distributed according to the density $\mu_\theta^N(\bx) \nu_N(\theta)$, but it is multi-modal and the corresponding Langevin dynamics is very metastable. As discussed in Section~\ref{sec:intro_twoscale}, in good situations, the multi-modality is entirely encoded in the marginal distribution $\nu_N$ of the collective variable $\xi(\bX,\Theta):=\Theta$. Hence, we target a modified measure where the conditional law of $\bX$ given $\xi(\bX,\Theta)=\theta$ is unchanged (this is still $\mu_\theta^N$) while the collective variable $\Theta$ is sampled at a higher temperature ($\beta^{-1}$ instead of $N^{-1}$).

However, if we naively apply standard collective variable-based enhanced sampling methods to our situation, then the initial target marginal distribution $\nu_N$ gets steeper and steeper as $N\rightarrow \infty$, and thus the  adaptive drift required to bias it to $\nu_\beta$ with a fixed $\beta$ becomes larger and larger, leading to unstable numerical schemes.  The reason why this doesn't occur for the two-scale sampler~\eqref{eq:thetaXcouples} is mainly due to the stochastic gradient approximation~\eqref{loc:approx}, which explicitly rely on the mean-field asymptotic structure of the problem and on the time-scale separation between the macroscopic and microscopic evolutions.

\subsection{Some definitions and notations}\label{subsec:def-not}

We write $\mathcal P_p(\R^d)$ the set of probability measures with finite $p^{th}$ moment, and $\mathcal W_p$ the associated Wasserstein distance.

A linear functional derivative of  $\mathcal E:\mathcal P_2(\R^d) \rightarrow (-\infty,\infty]$ is  a measurable function $\frac{\delta\mathcal{E}}{\delta m}:\mathcal P_2(\R^d)\times\R^d \rightarrow \R$ such that, for all  $\mu_1,\mu_0 \in \mathcal P_2(\R^d)$ with $\mathcal{E}(\mu_0)+\mathcal{E}(\mu_1)<\infty$,
\begin{equation}
    \label{eq:defFunctionDerivative}
\mathcal E(\mu_1) - \mathcal E(\mu_0) = \int_0^1 \int_{\R^d} \frac{\delta  \mathcal E}{\delta m}(t\mu_1 +(1-t) \mu_0,x) (\mu_1-\mu_0)(\dd x)\dd t\,.
\end{equation}
 If it exists, it is unique up to an additive constant. If $x \mapsto \frac{\delta \mathcal E}{\delta m}(\mu,x)$ is $\mathcal C^1$ for all $\mu \in \mathcal P_2(\R^d)$, then we call
 \begin{equation}
 \label{def:intrinsic}
 D\mathcal E(\mu,x) := \na_x \frac{\delta\mathcal E}{\delta m}(\mu,x)
 \end{equation}
 the intrinsic derivative of $\mathcal E$. 
 
 For two probability measures $\nu,\mu$ on $\R^d$ with $\nu \ll \mu$, the relative entropy and Fisher information of $\nu$ with respect to $\mu$ are defined by
 \begin{equation}
 \label{eq:HI}
 \mathcal H(\nu|\mu) = \int_{\R^d} \ln \po \frac{\dd \nu}{\dd \mu}\pf \dd \nu\,,\qquad \mathcal I(\nu|\mu) = \int_{\R^d}\left|\na  \ln \po  \frac{\dd \nu}{\dd \mu}\pf \right|^2 \dd \nu\,,
 \end{equation}
 where $|\na \ln(\dd \nu/\dd \mu)|$ is a classical gradient norm if $\dd \nu/\dd \mu$ is differentiable and is always defined in general as an upper gradient (see \cite[Definition 1.2.4]{ambrosio2005gradient}). If $\nu$ has no density with respect to $\mu$, $\mathcal H(\nu|\mu) = \mathcal I(\nu|\mu) = +\infty$.
 
\section{Theoretical results}\label{sec:theory}

There are two approximations between the scheme~\eqref{eq:thetaXcouples} and the target macroscopic dynamics: the particle one (depending on $N$) and the finite-speed  one (depending on $\varepsilon$). To check the consistency of the method, we provide successively  basic convergence guarantees for both errors.

\subsection{Propagation of chaos}

At fixed $\varepsilon>0$, we expect propagation of chaos to hold conditionally to the macroscopic noise $B$ as $N\rightarrow \infty$, so that $(\pi_{\bX_t},\theta_t)$ should converge to
\begin{equation}
\label{eq:thetamucouples}
\left\{
\begin{array}{rcl}
\partial_t \nu_t &=&  \na\cdot \po \nu_t\co  D\mathcal E_0(\nu_t,\cdot) - \bar\Theta_t\cdot \na \varphi\cf \pf  + \Delta\nu_t  \\
\dd \bar\Theta_t &=& - \varepsilon\po \bar\Theta_t - \nu_t(\varphi)\pf  + \sqrt{2\varepsilon/\beta} \dd B_t\,,
\end{array}\right.
\end{equation}
assuming that $\pi_{\bX_0} \rightarrow \nu_0$, which for instance holds if $\bX_0 \sim \nu_0^{\otimes N}$ (which we assume in Theorem~\ref{thm:chaos} to have a simpler statement). We work under these regularity conditions:


\begin{assumption}\label{assu:PoC}
\
\begin{enumerate} 
\item The intrinsic derivative $D\mathcal E_0$ exists. There exists $L>0$ such that, first, $x\mapsto D\mathcal E_0 (\nu,x)$ is $L$-one-sided Lipschitz continuous uniformly in $\nu$, meaning that
\begin{equation}
\label{eq:onesidedLip}
\forall \nu\in\mathcal P_2(\R^d),\ \forall x,y\in \R^{d},\qquad (x-y) \cdot \po D\mathcal E_0(\nu,x) - D\mathcal E(\nu,y)\pf \leqslant L |x-y|^2\,,
\end{equation}
and, second, $\mu \mapsto D\mathcal E_0(\mu,x)$ is $L$-Lipschitz continuous uniformly in $x$, i.e.
\begin{equation}
\label{eq:Lipschitzwrtmu}
\forall \nu,\mu \in \mathcal P_2(\R^d),\ \forall x\in\R^d,\qquad |D\mathcal E_0(\nu,x) - D\mathcal E_0(\mu,x)|\leqslant L \mathcal W_2(\nu,\mu)\,.
\end{equation}
\item The parameter $\varphi\in \mathcal C^1(\R^d,\R^m)$ is  Lipschitz continuous, and so is $\na \varphi$.
\end{enumerate} 
\end{assumption}

Under these regularity conditions, well-posedness for~\eqref{eq:thetaXcouples} and~\eqref{eq:thetamucouples} follows from classical fixed-point arguments. Applying then a  standard synchronous coupling arguments, we obtain the following (conditional) mean-field limit:

\begin{theorem}\label{thm:chaos}
Under Assumption~\ref{assu:PoC}, fix a Brownian motion $B$ on $\R^m$, an initial distribution $\nu_0 \in \mathcal P_p(\R^d)$ for some $p>2$ and $\theta_0 \in \R^m$. Denote by $(\nu_t,\bar\Theta_t)_{t\geqslant 0}$ the corresponding solution of~\eqref{eq:thetamucouples} (with $\bar\Theta_0=\theta_0$).  For all $N\geqslant 1$, let $\bX_0 \sim \nu_0^{\otimes N}$ (independent from $B$), let $B^1,\dots,B^N$ be $N$ independent $d$-dimensional Brownian motions independent from $B$ and $\bX_0$, and denote by $(\bX_t,\Theta_t)_{t\geqslant 0}$ the corresponding solutions of~\eqref{eq:thetaXcouples} (with $\Theta_0=\theta_0$). Then, for all $T>0$ and $\delta>0$,
\[\sup_{t\in[0,T]} \mathbb P \po \mathcal W_2(\pi_{\bX_t}, \nu_t)  + |\Theta_t - \bar\Theta_t| \geqslant \delta \pf \underset{N\rightarrow 0}\longrightarrow 0\,.\]
\end{theorem}

If we consider a convergence conditionally to the sigma algebra generated by $B$, then a more quantitative result is stated in Remark~\ref{rem:a.s.CV} below.

\subsection{Averaging}

In this section, with respect to~\eqref{eq:thetamucouples}, we accelerate time by a factor $1/\varepsilon$. For a fixed Brownian motion $B$ and an initial condition $(\nu_0,\Theta_0)$, we consider for all $\varepsilon>0$ the process
\begin{equation}
\label{eq:thetamucouples-epsilon}
\left\{
\begin{array}{rcl}
\partial_t \nu_t^\varepsilon &=&  \frac1\varepsilon \na\cdot \po \nu_t^\varepsilon\co  D\mathcal E_0(\nu_t^\varepsilon,\cdot) -  \Theta_t^\varepsilon \cdot \na \varphi\cf \pf  + \frac1\varepsilon \Delta\nu_t^\varepsilon  \\
\dd  \Theta_t^\varepsilon &=& -  \Theta_t^\varepsilon + \nu_t^\varepsilon(\varphi)   + \sqrt{2/\beta} \dd B_t\,.
\end{array}\right.
\end{equation}
We are interested in the limit of this process as $\varepsilon $ vanishes. We expect to observe an averaging phenomenon in the nice situation where, for any $\theta\in\R^m$, the equation 
\begin{equation}
\label{eq:nu_theta_fixed}
\partial_t \rho_t =  \na\cdot \po \rho_t\co  D\mathcal E_0(\rho_t,\cdot) -  \theta \cdot \na \varphi\cf \pf  + \Delta\rho_t 
\end{equation}
admits a unique stationary solution $\mu_\theta$, globally attractive. In this case, as $\varepsilon\rightarrow 0$, $\Theta^\varepsilon$ should converge to the solution of
\begin{equation}
\label{eq:Theta0}
\dd  \Theta_t^0 =  -  \Theta_t^0 + \mu_{\Theta_t^0}(\varphi)   + \sqrt{2/\beta} \dd B_t\,,
\end{equation}
with $\Theta_0^0 = \Theta_0$, while $\nu_t^\varepsilon$ should converge to $\mu_{\Theta_t^0}$.  This is the content of Theorem~\ref{thm:averaging}. Before stating it, let us motivate its assumptions and explain why~\eqref{eq:Theta0} is a Langevin process associated to some coarse-grained energy $\omega$ (as in~\eqref{eq:ThetaLangevinintro}, but now with temperature $\beta^{-1}$ instead of $N^{-1}$), without trying to state precise formal results in general settings (we refer to \cite{liu2020large,GuillinWuZhang,Pavliotis,Monmarchemetastable} for further details, and particularly to~\cite[Section 6.1]{MonmarcheEquivalence}).

Since~\eqref{eq:nu_theta_fixed} is the Wasserstein gradient flow of the free energy
\[\mathcal F_\theta(\nu) = \mathcal E_0(\nu) - \theta\cdot \nu(\varphi) + \mathcal H(\nu)\,,\]
 it admits a unique globally attractive stationary solution  if $\mathcal F_\theta$ admits a unique global minimiser and satisfies a so-called Polak-Łojasiewicz (PL) inequality, namely if there exists $\lambda_\theta>0$ such that
\begin{equation}
\label{eq:PLFtheta}
\forall \nu \in \mathcal P_2(\R^d), \qquad \tF_{\theta}(\nu) \leqslant \frac{1}{2\lambda_\theta} \mathcal I_{\theta}(\nu)\,,
\end{equation}
with 
\[\tF_\theta(\nu) = \mathcal F_\theta(\nu) - \inf \mathcal F_\theta \]
and the free energy dissipation
\[\mathcal I_{\theta}(\nu) := \int_{\R^d} \left|\na \ln \nu+ D\mathcal E_0(\nu)-\theta \cdot \na \varphi \right|^2\nu \]
(equal to $+\infty$ if $\nu$ doesn't have a Lebesgue density).  
From Laplace-Varadhan lemma, as $N\rightarrow \infty$, the macroscopic potential $\omega_N$ in~\eqref{eq:omegaN} converges to
\[\omega(\theta) = \frac12|\theta|^2 - \inf_{\nu \in \mathcal P_2(\R^d)} \mathcal F_\theta \,, \]
see e.g. \cite[(3.30)]{liu2020large}. Moreover, assuming suitable conditions, under the Gibbs measure $\mu_\theta^N$, $\pi_{\bx}$ almost surely converges to $\mu_\theta$. As a consequence,   letting $N\rightarrow\infty$ in~\eqref{eq:naomegaN} gives
\begin{equation}\label{loc:zdfdsgfhf}
\na \omega(\theta) = \theta-\mu_{\theta}(\varphi)\,,    
\end{equation}
meaning that~\eqref{eq:Theta0} is indeed a Langevin process with potential $\omega$ and temperature $\beta^{-1}$.

As in Section~\eqref{sec:intro_twoscale}, the critical points of $\omega$ (i.e.  the fixed points of $\theta \mapsto \mu_\theta(\varphi)$) are in one-to-one correspondence with the stationary solutions of~\eqref{eq:gradientflowF} (by $\theta \leftrightarrow \mu_\theta $), with local minimisers of $\omega$ corresponding to local minimisers of $\mathcal F$. In particular, $\mu_{\Theta_t^0}$ does sample efficiently these local minimisers (since the Gibbs density proportional to $e^{-\beta \omega}$ concentrates on the local minimisers of $\omega$).

By contrast, when $\mathcal F_\theta$ has several critical points, then most of the nice properties listed above fail, see e.g. \cite{Monmarchemetastable}. The limit behavior of~\eqref{eq:thetamucouples-epsilon} is then more intricate and we won't try to study this bad situation, as the algorithm is precisely not designed to work then (all the metastability should be incorporated in the macroscopic dynamics of $\Theta_t$). 

 As a conclusion of this discussion, we work under the following conditions, which, as stated in Proposition~\ref{prop:lambda_thetaPL} ensures that we are in the good situation discussed above.

 \begin{assumption}\label{assum:PL}\
 
 \begin{enumerate}
 \item The energy $\mathcal E_0$ is semi-lower continuous and $D\mathcal E_0$ exists.
 \item There exists $\eta \in[0,1/4)$ such that  for all $\nu,\nu'\in\mathcal P_2(\R^d)$ and $t\in[0,1]$
\begin{equation}
\label{eq:convexity-c}
\mathcal E_0(t\nu + (1-t) \nu') \leqslant t \mathcal E_0(\nu) +(1-t)\mathcal E_0(\nu') + t(1-t) \eta \|\nu-\nu'\|_{TV}^2 \,.
\end{equation}
\item There exists $c_0,C_0 >0$ such that for any $\nu\in\mathcal P_2(\R^d)$,
\begin{equation}
\label{eq:lowerboundE0m2}
\mathcal E_0(\nu) \geqslant c_0 \int_{\R^d} |x|^2 \nu(\dd x) - C_0\,.
\end{equation}
\item There exist $\kappa,M,L>0$ such that for all $\nu\in \mathcal P_2(\R^d)$, $\frac{\delta \mathcal E_0}{\delta m}(\nu,\cdot)$ is the sum of three functions: a $\kappa$-strongly convex one, a $M$-bounded one and a $L$-Lipschitz one. 
\item The parameter can be decomposed as $\varphi=(\varphi_{\ell},\varphi_b) \in \R^{m_\ell}\times \R^{m_b}$ (possibly with $m_\ell=0$ or $m_b=0$)  with $\varphi_\ell$  linear and $\varphi_b$  bounded.
 \end{enumerate} 
 \end{assumption}

\begin{proposition}\label{prop:lambda_thetaPL}
Under Assumption~\ref{assum:PL}, there exists $\lambda>0$ depending only on $\kappa,L,M,\eta$ such that for all $\theta=(\theta_\ell,\theta_b)\in \R^{m_\ell}\times\R^{m_b}$, $\mathcal F_\theta$ admits a unique global minimiser $\mu_\theta\in \mathcal P_2(\R^d)$ and satisfies the PL inequality~\eqref{eq:PLFtheta} with
\begin{equation}
\label{eq:lambdathetapropr2}
\lambda_\theta \geqslant \lambda e^{-|\theta_b|\|\varphi_b\|_\infty}\,. 
\end{equation}
\end{proposition} 

%
%
%
%
%

The main result of this section is then the following:

\begin{theorem}
\label{thm:averaging}
Under Assumption~\ref{assum:PL}, assume furthermore that $|\Theta_0|^2$ and $\mathcal F_{\Theta_0}(\nu_0)$ have finite expectation and that $\varphi$ is Lipschitz continuous. Then,  for any $T>t_0>0$ and $\delta>0$,
\begin{equation}
\label{eq:thmaveraging}
\sup_{t\in[0,T]} \mathbb P \po |\Theta_t^\varepsilon- \Theta_t^0|  \geqslant \delta   \pf + \sup_{t\in[t_0,T]} \mathbb P \po  \mathcal W_2(\nu_t^\varepsilon,\mu_{\Theta_t^0}) \geqslant \delta  \pf  \underset{\varepsilon\rightarrow 0}\longrightarrow 0\,.
\end{equation}
\end{theorem}

In the case where $\varphi$ is linear (i.e. $m_b=0$), we get a more quantitative result, cf. Remark~\ref{rem:2}. Moreover, in the case $m_b\neq 0$, this stronger result still holds for a truncated version of the process, given by~\eqref{eq:thetamucouples-epsilon-g}, which can possibly be of interest in practice in order to sample the equilibria $\mu_\theta$ with $\theta$ constrained in a ball.

\subsection{Example}\label{sec:example}

Consider an energy $\mathcal E$ with pairwise interaction of the form~\eqref{eq:ECurieWeiss-extended}, where $W$ admits a decomposition~\eqref{eq:Mercer}. For some $m \in\N$, set
\[\varphi(x) = \po n_k(x)\pf_{k\in\cco 0,m\ccf}\,,\]
and
\begin{equation}
\label{loc:examplesfdfed}
\mathcal E_0(\rho) =   \int_{\R^d} V_1(x)\rho(x) \dd x + \frac12\sum_{k\geqslant 0}  \po \int_{\R} m_k(x)\rho(x)\dd x\pf^2  - \frac12 \sum_{k>m} \po \int_{\R} n_k(x)\rho(x)\dd x\pf^2\,.
\end{equation}
Then the energy $\mathcal E$ is decomposed as~\eqref{eq:geneE}. In fact we can write
\begin{equation}
\label{loc:exampleqsdsf}
\mathcal E_0(\rho) = \int_{\R^d} V(x) \rho(\dd x) + \frac12 \int_{\R^d} W_1(x,y)\rho(\dd x)\rho(\dd y),\qquad W_1(x,y) = W(x,y) + \sum_{k=0}^m n_k(x)n_k(y)\,. 
\end{equation}

The conditions on $\varphi$ required in Assumptions~\ref{assu:PoC} and~\ref{assum:PL} are easily checked on specific models. For instance, for quadratic interaction as in Section~\ref{sec:intro_twoscale}, $\varphi(x) = a x$ for some $a\in\R$, hence $\varphi$ is linear and a fortiori Lipschitz continuous (and so is $\na \varphi$). For the attractive Gaussian interaction~\eqref{eq:Gaussianinteraction}
\[n_k(x) = \sqrt{\alpha} e^{-\frac{x^2}{2}} \frac{x^k }{\sqrt{k!}} \,, \]
so that $\varphi$ is bounded with bounded derivatives of all orders.

We now turn to the conditions on $\mathcal E_0$. First,
\[D\mathcal E_0(\rho,x) =  \na V_1(x) + \sum_{k\geqslant 0} \na m_k(x) \rho(m_k) - \sum_{k>m} \na n_k(x) \rho(n_k) \,.  \]
Besides, it can be simpler to check the conditions by writing
\begin{equation}
\label{loc:exempleqfze}
D\mathcal E_0(\rho,x) = D\mathcal E(\rho,x) + \na \varphi(x)\cdot \rho\po \varphi(x)\pf = \na V_1(x) + \int_{\R^d} \na_x W(x,y)\rho(\dd y) + \na \varphi(x)\cdot \rho\po \varphi\pf \,.
\end{equation}
Assume that $\na V_1$ satisfies a one-sided Lipschitz condition~\eqref{eq:onesidedLip} (which is the case for instance if $V_1$ is  convex outside a compact set). With the conditions we checked on $\varphi$, we already have that $x\mapsto \na \varphi(x)\cdot \rho(\varphi)$ satisfies this condition (for a linear function $f$, $(\na f(x) - \na f(y))\rho(f)$ is zero, and for a bounded function $f$ with second order derivative bounded, $|\na f(x) - \na f(y)||\rho(f)|\leqslant \|\na^2 f\|_\infty\|f\|_\infty|x-y|$). Then, if $\na_x^2 W$ is bounded (which is for instance the case for the Gaussian interaction~\eqref{eq:Gaussianinteraction}), the remaining term in~\eqref{loc:exempleqfze} is $\|\na_x^2 W\|_\infty$-Lipschitz in $x$ and  finally~\eqref{eq:onesidedLip} holds.

Similarly, assuming that $\na_{x,y}^2 W$ is bounded,
\[|D\mathcal E_0(\nu,x) - D\mathcal E_0(\mu,x)| \leqslant  \po \|\na_{x,y}^2 W\|_\infty + \|\na \varphi\|_\infty^2\pf  \mathcal W_2(\nu,\mu) \,, \]
so that~\eqref{eq:Lipschitzwrtmu} holds.

Under the conditions discussed above, we have thus already checked the conditions in Assumption~\ref{assu:PoC}. Concerning Assumption~\ref{assum:PL}, the lower-bounded~\eqref{eq:lowerboundE0m2} holds for instance if, in~\eqref{loc:exampleqsdsf},  we can decompose $W_1(x,y)= - |A(x-y)|^2 + W_2(x,y)$ for some matrix $A$ and some bounded function $W_2$  (which corresponds to the assumption we have already made on the functions $n_k,m_k$) and if $V(x) \geqslant c_0 |x|^2 + |Ax|^2 - C$ for some constants $c_0,C>0$. Moreover, with this form,
\[\frac{\delta\mathcal E_0}{\delta m}(\rho,x) = V(x) - \int_{\R^d} |A(x-y)|^2 \rho(\dd y) + \int_{\R^d} W_2(x,y)\rho(\dd y)\,.  \]
Assume that $V=V_c+V_b$ where $x\mapsto V_c(x) - |Ax|^2$ is $\kappa$-strongly convex and $V_b$ is $M$-bounded for some $\kappa,M>0$. Then, for all $\rho\in\mathcal P_2(\R^d)$, $\frac{\delta\mathcal E_0}{\delta m}(\rho,x)$ is the sum of $V_c(x) -  \int_{\R^d} |A(x-y)|^2 \rho(\dd y)$, which is $\kappa$-strongly convex as a function of $x$, and of $V_b + \int_{\R^d} W_2(\cdot,y)\rho(\dd y)$, which is $M+\|W_2\|_\infty$-bounded. The fourth condition of Assumption~\ref{assum:PL} is thus satisfied.

Finally, since the two first terms in the right hand side of~\eqref{loc:examplesfdfed} are (flat) convex as a function of $\rho$, they do not contribute when checking~\eqref{eq:convexity-c}, and we see that for all $\nu,\nu'\in\mathcal P_2(\R^d)$ and $t\in[0,1]$,
\begin{align*}
\mathcal E_0(t\nu + (1-t) \nu') -  t \mathcal E_0(\nu) - (1-t)\mathcal E_0(\nu') & \leqslant  t(1-t) \sum_{k>m} \po \nu(n_k)-\nu'(n_k)\pf^2 \\
& = t(1-t)   \sum_{k>m} \po  \mathbb E  \co n_k(Y)-n_k(Y')\cf \pf ^2\,,
\end{align*}
with $(Y,Y')$ a coupling of $(\nu,\nu')$. Taking an optimal coupling for the total variation distance, in the sense that $\|\nu-\nu'\|_{TV} =2\mathbb P(Y\neq Y')$, 
\begin{align*}
\sum_{k>m} \po  \mathbb E  \co n_k(Y)-n_k(Y')\cf \pf ^2 &\leqslant \frac14 \|\nu-\nu'\|_{TV}^2 \sum_{k>m} \po  \mathbb E  \co n_k(Y)-n_k(Y') \ | \ Y\neq Y'\cf \pf ^2
\\
&\leqslant \frac14 \|\nu-\nu'\|_{TV}^2  \mathbb E  \co \left. \sum_{k>m} \po  n_k(Y)-n_k(Y')\pf ^2 \ \right| \ Y\neq Y'\cf  \\
&\leqslant \|\nu-\nu'\|_{TV}^2\left\| \sum_{k>m} n_k^2 \right \|_\infty\,.
\end{align*}
This shows that~\eqref{eq:convexity-c} holds with
\[\eta = \left\| \sum_{k>m} n_k^2 \right \|_\infty\] 
If we can take $m$ large enough so that $\eta<1/2$ (which is the case in particular when $\sum_{k\geqslant 0} \|n_k\|_\infty^2< \infty$) then Assumption~\ref{assum:PL} is satisfied.


\section{Numerical experiments}\label{sec:numerique}

In this section we consider models where the temperature $\sigma^2$ in~\eqref{eq:energieLibre} is not necessarily equal to $1$. When implementing the coupled process~\eqref{eq:thetaXcouples}, we apply it to the free energy $\sigma^{-2}\mathcal F = \sigma^{-2}\mathcal E + \mathcal H$, so that in~\eqref{eq:geneE} $\varphi$ has to be replaced by $\varphi/\sigma$; we then rescale $\Theta$ into $\Theta/\sigma$ and $\varepsilon$ into $\varepsilon/\sigma^2$, so that the final rescaled process reads
\begin{equation}
\label{eq:rescaled-coupled}
\left\{
\begin{array}{rcl}
\forall i\in\cco 1,N\ccf,\qquad \dd X_t^i &=& - \sigma^{-2}\na_{x_i} V_N(\bX_t) \dd t + \Theta_t \cdot \na \varphi(X_t^i) \dd t + \sqrt{2}\dd B_t^i\\
\dd \Theta_t &=& - \varepsilon\po \sigma^2\Theta_t - \frac1N\sum_{i=1}^N \varphi(X_t^i)\pf \dd t + \sqrt{2\varepsilon/\beta} \dd B_t\,.
\end{array}\right.
\end{equation}
In all experiments, it is simulated using an Euler--Maruyama scheme with step-size $h=0.01$.

\subsection{Double-well potential}
\label{subsec:double-well}

We first consider a Curie--Weiss model corresponding to the quadratic energy~\eqref{eq:self-consistent} in dimension 1 with
\[
V(x) = \frac{x^4}{4} - \frac{x^2}{2}\,, \qquad W(x-y) = \frac{\lambda}{2}\,|x-y|^2\,, \quad \lambda = 1\,,
\]
so that the feature reduces to $\varphi(x)=x$ (hence $m=1$ and $\Theta_t\in\R$). There is a critical temperature $\sigma_c^2 = 0.68^2 \approx 0.46$ such that: for $\sigma^2 > \sigma_c^2$ the free energy $\mathcal F$ admits a unique minimiser, a symmetric density centred at $0$; while for $\sigma^2 < \sigma_c^2$ it admits two minimisers $\mu_+$ and $\mu_-$, symmetric of one another, given by the self-consistency equation~\eqref{eq:self-consistent} and with opposite barycentres $\pm m_*$ (numerically, $m_* \approx 0.79$ at the sub-critical temperature $\sigma^2=0.25$ used in the experiments below). Throughout, the reference measures $\mu_\pm$ are obtained by relaxing a standard Langevin dynamics~\eqref{eq:LangevinN} from Gaussian initial conditions $\mathcal N(\pm 1, 0.5)$.

\paragraph{The two stationary regimes.}
Before applying the two-scale sampler, we illustrate the two stationary regimes relying on the plain mean-field particle system~\eqref{eq:LangevinN}, that is, the standard interacting Langevin dynamics associated with $\mathcal F$ \emph{without} the macroscopic variable $\Theta$. For $N$ particles it reads
\begin{equation}
\label{eq:dw-uncoupled}
\dd X_t^i = -V'(X_t^i)\dd t - \lambda\po X_t^i - \bar X_t\pf\dd t +\sqrt{2\sigma^2}\,\dd B_t^i\,,\qquad \bar X_t = \frac1N\sum_{j=1}^N X_t^j\,,
\end{equation}
for $i\in\cco 1,N\ccf$, where $V'(x) = x^3 - x$, the interaction strength is $\lambda = 1$, and $\sigma^2$ is the temperature. We integrate~\eqref{eq:dw-uncoupled} with an Euler--Maruyama scheme of step $h = 0.01$, running it independently from two i.i.d. initial distributions, $X_0^i \sim \mathcal N(+1, 0.5)$ and $X_0^i \sim \mathcal N(-1, 0.5)$, with $N = 5000$ particles up to $T = 100$. For each run we record the final empirical distribution of the particles and the barycentre trajectory $t\mapsto \bar X_t$.

The outcome depends on whether the temperature lies below or above $\sigma_c^2$. In the sub-critical regime $\sigma^2 = 0.25 < \sigma_c^2$ (Figure~\ref{fig:dw-basic-sub}), the two runs relax towards two distinct stationary solutions $\mu_+$ and $\mu_-$, symmetric of one another, whose barycentres settle at $\approx +0.79$ and $\approx -0.79$: the initial condition alone selects which of the two minimisers is reached, and a trajectory started near one of them remains there. In the super-critical regime $\sigma^2 = 0.64 > \sigma_c^2$ (Figure~\ref{fig:dw-basic-super}), both runs relax, whatever their starting point, to the same symmetric solution centred at $0$. As we see, only the sub-critical case corresponds to the metastable situation we are interested in: below $\sigma_c^2$ the free energy has several minimisers, and the plain dynamics cannot move from one to another on a reasonable time scale. We focus on this regime in the rest of the section.

\begin{figure}[htbp]
    \centering
    \includegraphics[width=0.9\linewidth]{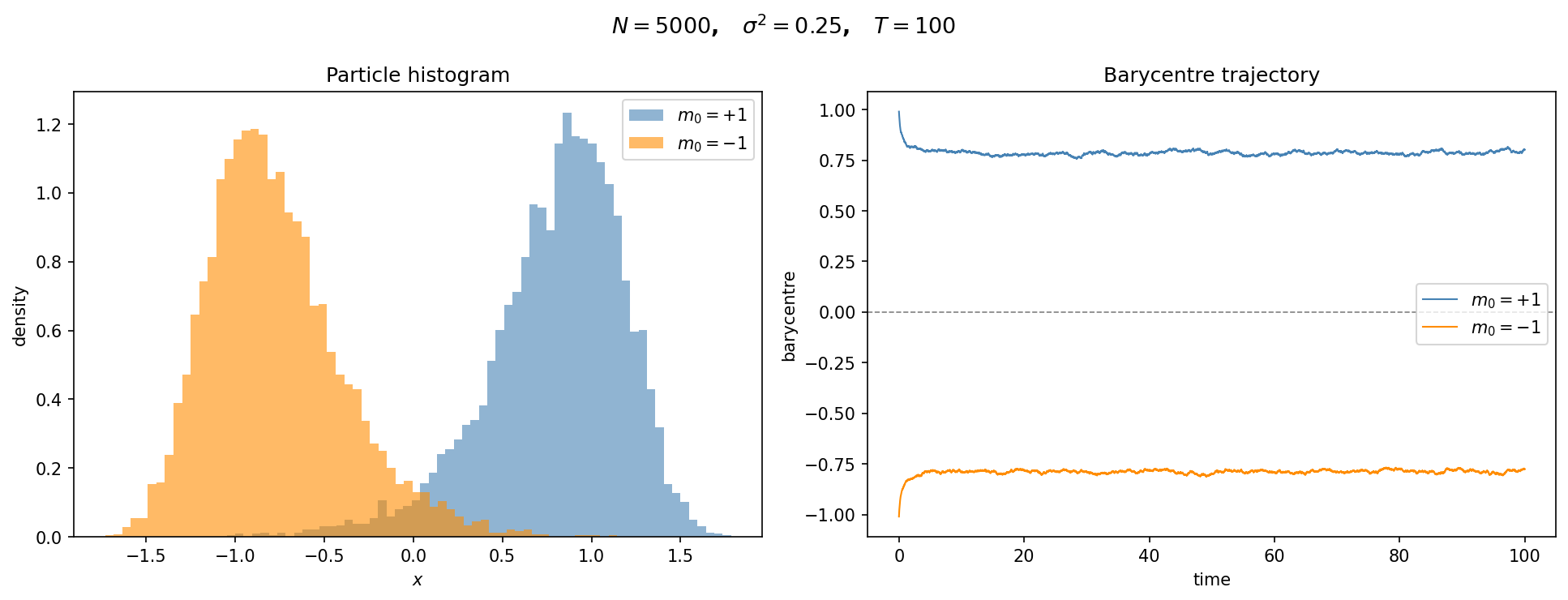}
    \caption{Sub-critical regime, $\sigma^2 = 0.25 < \sigma_c^2$. Uncoupled system~\eqref{eq:dw-uncoupled} with $N = 5000$, $T = 100$, started from $\mathcal N(\pm1, 0.5)$. Left: final histograms of the particles for the two initial conditions $m_0 = \pm 1$, concentrating around the two stationary solutions $\mu_\pm$. Right: the two barycentre trajectories, settling at $\pm m_* \approx \pm 0.79$.}
    \label{fig:dw-basic-sub}
\end{figure}

\begin{figure}[htbp]
    \centering
    \includegraphics[width=0.9\linewidth]{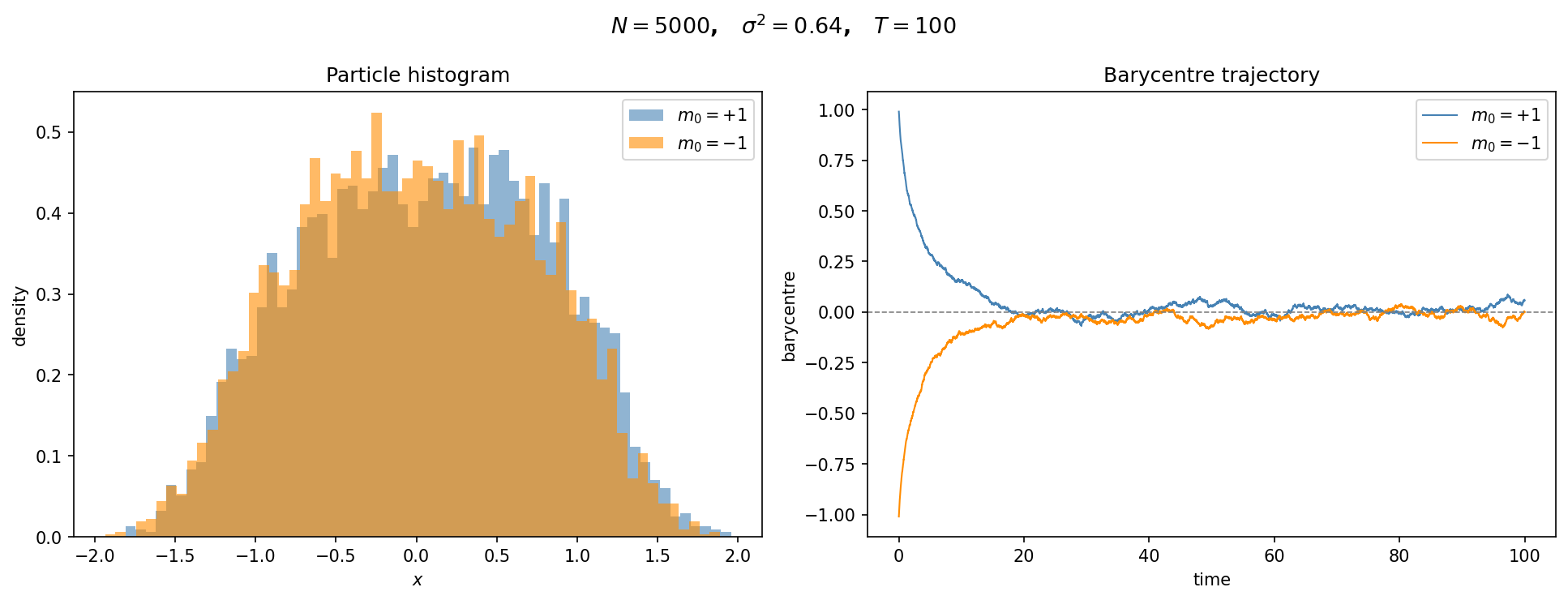}
    \caption{Super-critical regime, $\sigma^2 = 0.64 > \sigma_c^2$. Same setting as Figure~\ref{fig:dw-basic-sub}. From either initial condition the particles relax to the single stationary solution centred at $0$, and the two barycentre trajectories converge to $0$.}
    \label{fig:dw-basic-super}
\end{figure}

\paragraph{Transitions, Wasserstein distance and the barycentre as a proxy.}
We now turn on the coupling between the particles and an additional macroscopic variable, simulating the two-scale system~\eqref{eq:thetaXcouples} for the double well. Taking $\varphi(x)=x$ (so $\Theta_t\in\R$) and $V_N = \sum_i V_1(X_t^i)$ with $V_1 = V + x^2/2$, and using $V_1'(x)=x^3$, the coupled system reads
\begin{equation}
\label{eq:dw-coupled}
\dd X_t^i = -\frac{(X_t^i)^3}{\sigma^2}\,\dd t + \Theta_t\,\dd t + \sqrt{2}\,\dd B_t^i\,,\qquad
\dd \Theta_t = -\varepsilon\po \sigma^2 \Theta_t - \bar X_t\pf\,\dd t + \sqrt{2\varepsilon/\beta}\,\dd W_t\,,
\end{equation}
for $i\in\cco 1,N\ccf$, where the mean-field interaction of the previous paragraph is now \emph{replaced} by the slow variable $\Theta_t$, which tracks the barycentre $\bar X_t$ at speed $\varepsilon$. We integrate~\eqref{eq:dw-coupled} with the same Euler--Maruyama scheme, and compare the particles to the two reference solutions $\mu_+$ and $\mu_-$ obtained as above. To locate the transitions we use two indicators: the quadratic Wasserstein distances $\mathcal W_2(\pi_{\bX_t},\mu_\pm)$ from the empirical distribution of the particles to the references (which we approximate by the empirical distribution of the particles from Figure~\ref{fig:dw-basic-sub}), and the barycentre $\bar X_t$ itself. Recall that, for two probability measures $\mu,\nu$ on $\R^d$,
\[
\mathcal W_2(\mu,\nu)^2 = \inf_{\gamma \in \Pi(\mu,\nu)} \int_{\R^d\times\R^d} |x-y|^2 \, \gamma(\dd x,\dd y)\,,
\]
where $\Pi(\mu,\nu)$ is the set of couplings of $\mu$ and $\nu$; in dimension $d=1$, for two empirical measures carried by $N$ points, this reduces to the sorted $L^2$ distance $\mathcal W_2(\pi_{\bx},\pi_{\by})^2 = \frac1N\sum_{i=1}^N (x_{(i)}-y_{(i)})^2$, with $x_{(1)}\leqslant\dots\leqslant x_{(N)}$ and $y_{(1)}\leqslant\dots\leqslant y_{(N)}$ the order statistics.

Figure~\ref{fig:dw-w2} shows a run at $\sigma^2 = 0.25$, $\varepsilon = 0.07$, $\beta = 0.4$, $N = 3000$, $T = 500$, displaying four quantities: the histograms of the particles conditioned on the current basin (top left), the barycentre $\bar X_t$ (top right), the trajectory of $\Theta_t$ (bottom left), and the two distances $\mathcal W_2(\pi_{\bX_t},\mu_\pm)$ (bottom right). Here conditioning on the current basin means that, at each recorded time, the population is assigned to $\mu_+$ or $\mu_-$ according to the sign of the barycentre $\bar X_t$, and the particle positions are pooled separately for the two cases; the two resulting histograms are then compared to the references $\mu_\pm$. Driven by $\Theta_t$, the population moves back and forth between the two basins, spending here about $68\%$ of the time in $\mu_+$ and $32\%$ in $\mu_-$. The two conditional histograms only \emph{resemble} the references $\mu_\pm$ and are noticeably broader: pooling every recorded time spent in a basin describes the fluctuations (at temperature $\beta^{-1}$, not vanishing as $N\rightarrow \infty$) around the minimiser.  
The two $\mathcal W_2$ distances exchange cleanly at each transition. Crucially, the barycentre $\bar X_t$ carries exactly the same transition information as the pair of $\mathcal W_2$ distances, at a negligible cost: this validates the barycentre as a cheap proxy for detecting transitions, which is relevant in high dimension where $\mathcal W_2$ becomes prohibitive to compute.

\begin{figure}[htbp]
    \centering
    \includegraphics[width=0.95\linewidth]{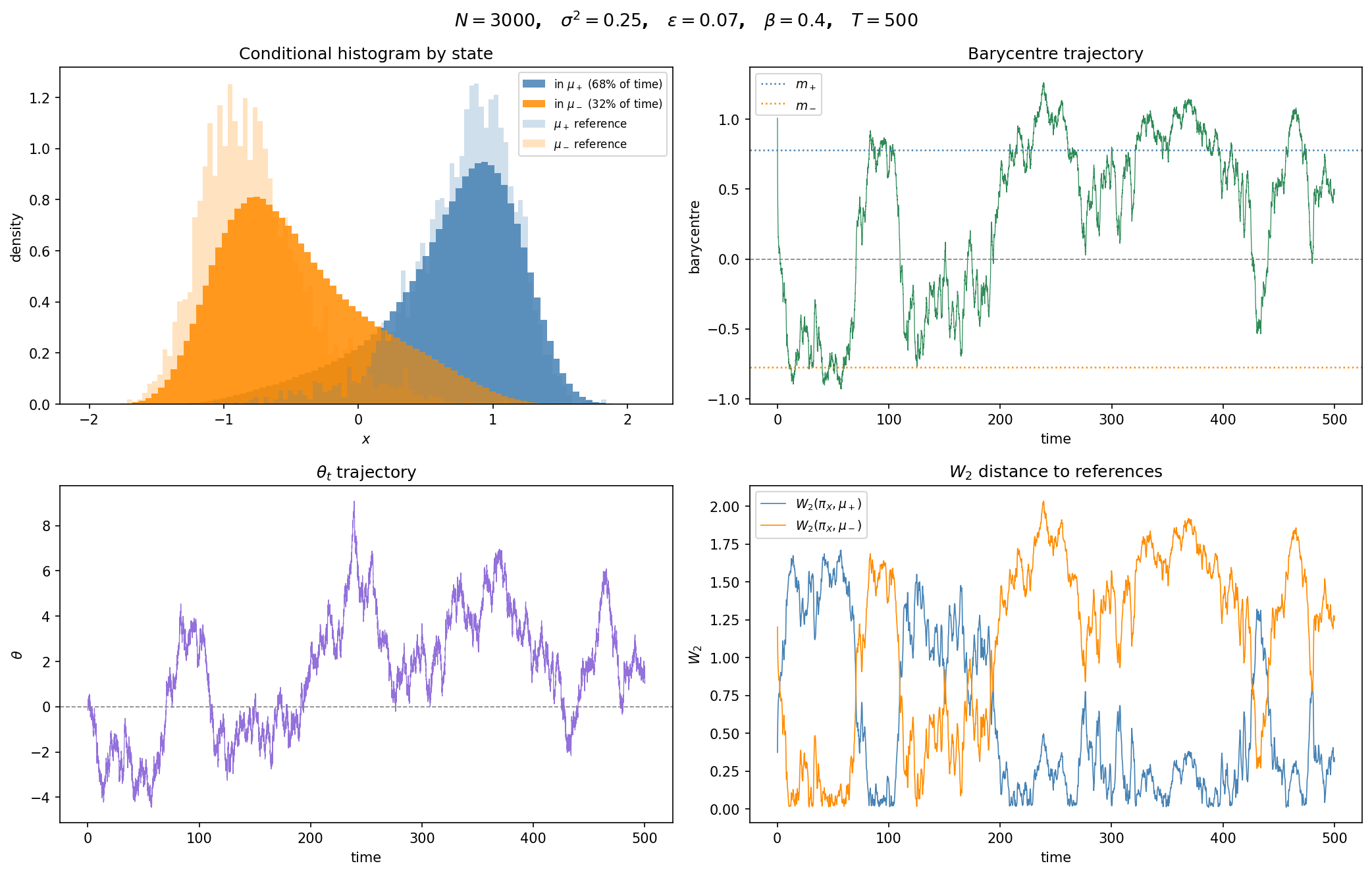}
    \caption{Coupled two-scale dynamics~\eqref{eq:dw-coupled} at $\sigma^2 = 0.25$, $\varepsilon = 0.07$, $\beta = 0.4$, $N = 3000$, $T = 500$. Top left: histograms of the particles conditioned on the current basin, compared to the references $\mu_\pm$ (light). Top right: barycentre $\bar X_t$, with the reference levels $m_\pm$. Bottom left: trajectory of $\Theta_t$. Bottom right: distances $\mathcal W_2(\pi_{\bX_t},\mu_\pm)$. The barycentre and the $\mathcal W_2$ distances signal the same transitions.}
    \label{fig:dw-w2}
\end{figure}

\paragraph{Influence of $\beta$ on the transitions.}
The macroscopic variable $\Theta_t$ follows a Langevin dynamics at inverse temperature $\beta$ in the coarse-grained potential: the noise intensity $\sqrt{2\varepsilon/\beta}$ decreases with $\beta$, so $\beta$ controls the frequency of the transitions. Figure~\ref{fig:dw-beta} compares three runs of~\eqref{eq:dw-coupled} at fixed $\varepsilon = 0.03$, $N = 1500$, $T = 500$, $\sigma^2 = 0.25$, for $\beta \in \{0.05,\, 0.6,\, 2\}$. Each row shows the trajectory of $\Theta_t$ (left), its histogram together with the fraction of time spent on each side of $\Theta = 0$ (centre), and the distances $\mathcal W_2(\pi_{\bX_t},\mu_\pm)$ to the two references (right). On the $\Theta$ axis the two basins are separated by $\Theta = 0$ and centred at the deterministic fixed points $\theta^*_\pm = m_\pm/\sigma^2 \approx \pm 3.1$.

At small $\beta = 0.05$ the macroscopic noise is strong and $\Theta_t$ crosses very frequently between the two basins, visiting both in comparable proportion. As $\beta$ increases the crossings become progressively rarer: at $\beta = 0.6$, and even more at $\beta = 2$, the trajectory dwells for long stretches in one basin and only switches occasionally.   The number of transitions thus decreases as $\beta$ increases --- the expected metastability/temperature trade-off, now carried by the one-dimensional variable $\Theta_t$ rather than by the full particle system --- so that $\beta$ tunes how aggressively the sampler explores the different minimisers.

\begin{figure}[htbp]
    \centering
    \includegraphics[width=0.9\linewidth]{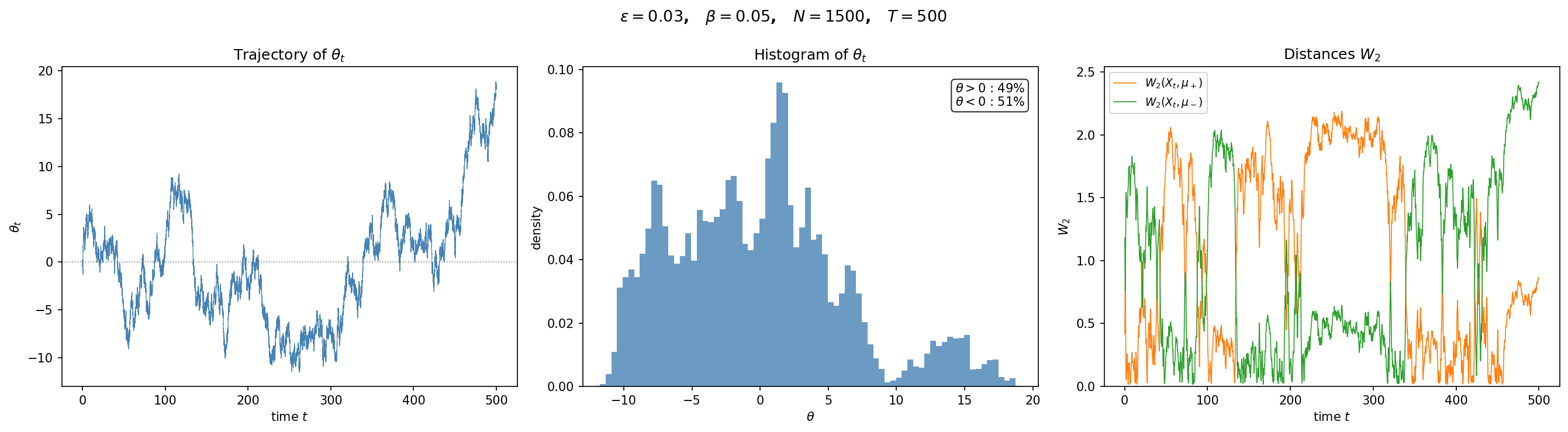}\\[2pt]
    \includegraphics[width=0.9\linewidth]{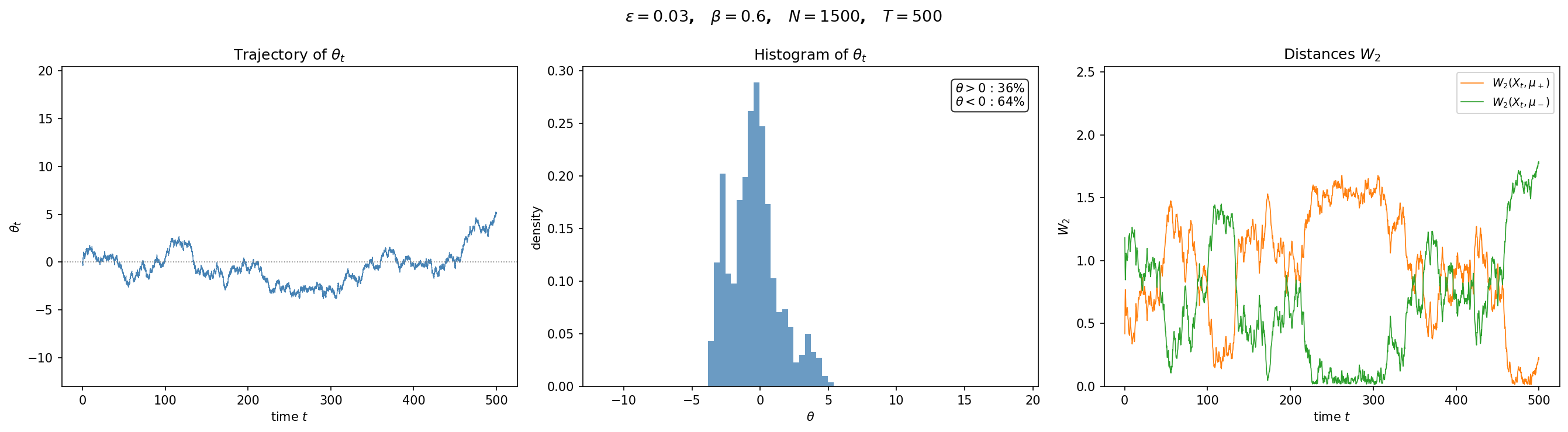}\\[2pt]
    \includegraphics[width=0.9\linewidth]{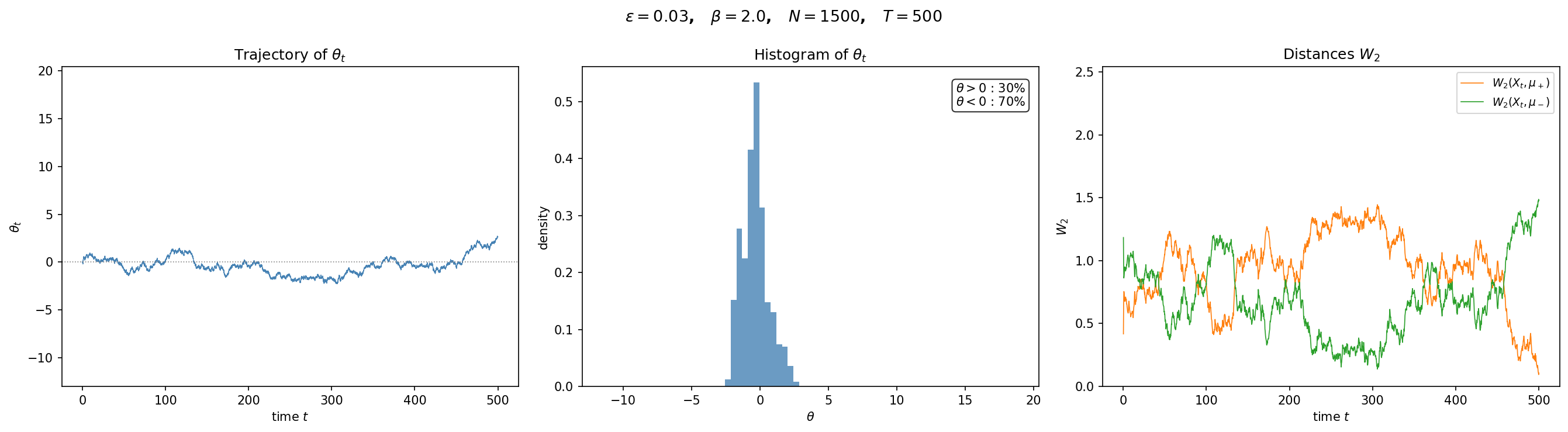}
    \caption{Influence of $\beta$ at fixed $\varepsilon = 0.03$, $N = 1500$, $T = 500$, $\sigma^2 = 0.25$. From top to bottom: $\beta = 0.05$, $0.6$, $2$. Each row shows the trajectory of $\Theta_t$ (left), its histogram together with the fraction of time spent on each side of $\Theta=0$ (centre), and the distances $\mathcal W_2(\pi_{\bX_t},\mu_\pm)$ (right). The number of transitions decreases as $\beta$ increases.}
    \label{fig:dw-beta}
\end{figure}

\paragraph{Long-time distribution of $\Theta_t$ versus the target $\nu_\beta$.}
The previous runs show $\Theta_t$ crossing between the basins; we now check  that its long-time law is the one predicted by the averaging principle. When $\varepsilon$ is small, $\Theta_t$ evolves slowly enough for the particles to equilibrate at each frozen value of $\Theta$, and its stationary law should be the coarse-grained Gibbs measure $\nu_\beta \propto e^{-\beta\omega}$. In the present quadratic case $\omega$ is explicit and we evaluate it by numerical quadrature,
\[
\omega(\theta) = \frac{\sigma^2}{2\lambda}\,\theta^2 - \ln \int_{\R} \exp\po -\frac{1}{\sigma^2}\po V(x) + \frac{\lambda}{2}x^2\pf + \theta x \pf \dd x\,,
\]
with $\lambda = 1$. We integrate~\eqref{eq:dw-coupled} over $M = 20$ independent trajectories, with i.i.d.\ initial conditions $X_0^i \sim \mathcal N(-1, 0.5)$ and $\Theta_0 = \bar X_0$, pool all the recorded values of $\Theta_t$ into a single histogram (the several trajectories improve the statistics, and the horizon is long enough for $\Theta_t$ to visit both basins), and compare it to $\nu_\beta$. Figure~\ref{fig:dw-nu} does this at $\sigma^2 = 0.3$, $\beta = 0.5$, $N = 1000$, for two values of the speed parameter, $\varepsilon = 0.05$ (horizon $T = 10^4$) and $\varepsilon = 0.1$ (horizon $T = 10^3$). In both cases the empirical distribution matches the target $\nu_\beta$, and the agreement improves as $\varepsilon$ decreases (at horizon $T \propto 1/\varepsilon$), as expected since the averaging approximation becomes exact in the limit $\varepsilon \to 0$.

\begin{figure}[htbp]
    \centering
    \includegraphics[width=0.9\linewidth]{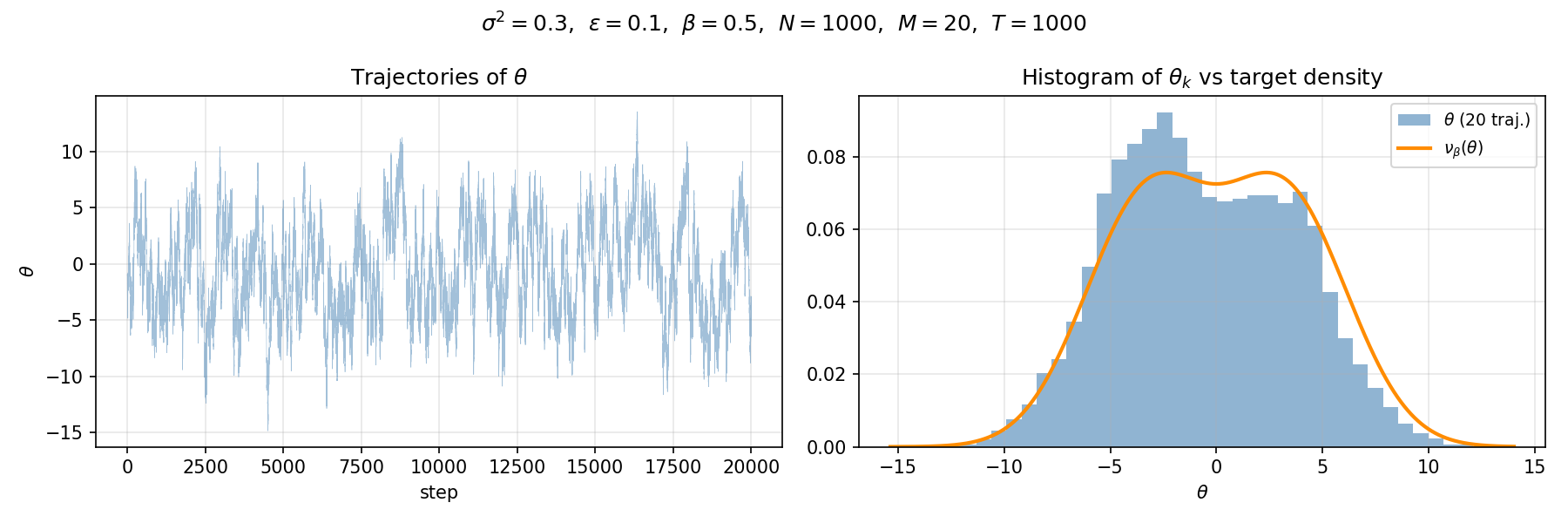}\\[2pt]
    \includegraphics[width=0.9\linewidth]{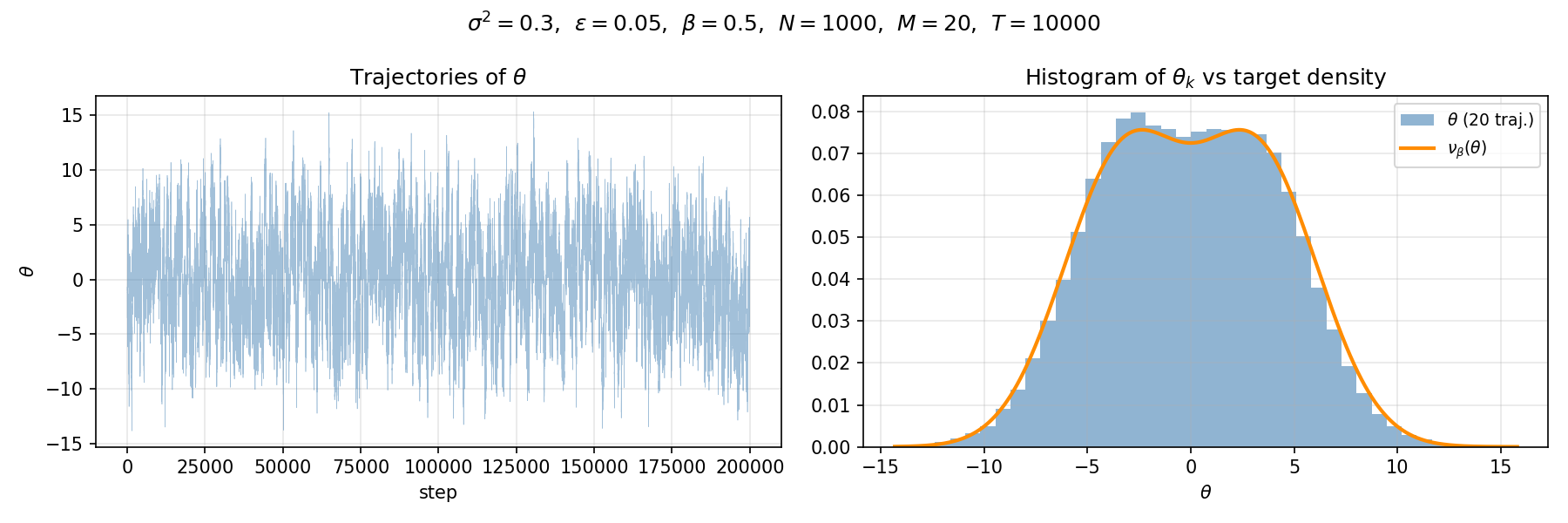}
    \caption{Long-time distribution of $\Theta_t$ at $\sigma^2 = 0.3$, $\beta = 0.5$, $N = 1000$, aggregated over $M = 20$ trajectories. Top: $\varepsilon = 0.05$, $T = 10^4$. Bottom: $\varepsilon = 0.1$, $T = 10^3$. Left: the $M$ trajectories of $\Theta_t$, concatenated one after another. Right: histogram of $\Theta_t$ against the target density $\nu_\beta \propto e^{-\beta\omega}$. }
    \label{fig:dw-nu}
\end{figure}

\subsection{Multi-well potential}
\label{subsec:multi-well}

We keep the quadratic interaction of Section~\ref{subsec:double-well} (feature $\varphi(x)=x$, so $\Theta_t\in\R$, and $\lambda = 1$), but replace the confining potential by the multi-well
\[
V(x) = \cos(x) + \cos\!\po \tfrac{x}{2}\pf + \frac{x^2}{100}\,.
\]
The difference with the double well is one of complexity rather than of nature. In the double-well case the potential $V$ already has two wells, and, below the critical temperature $\sigma_c^2$, the free energy has exactly two symmetric minimisers $\mu_\pm$, whereas above $\sigma_c^2$ it has a single one --- the whole picture being governed by a single critical temperature. Here $V$ has many wells of unequal depth, so the coarse-grained potential $\omega$ can have many  local minima and the free energy $\mathcal F$ several distinct localised stationary solutions $\mu_{\theta^*}$. On the range $[-30,30]$ the wells of $V$ sit at $x \approx \pm 3.57,\ \pm 8.73,\ \pm 15.89,\ \pm 20.97,\ \pm 28.21$ (ten wells), and more appear on a wider range. This is the genuinely multi-modal situation the two-scale sampler is designed for.

The specific choice of $V$ is deliberate. The two cosines $\cos x$ and $\cos(x/2)$ have different periods ($2\pi$ and $4\pi$), so their sum is not a regular array of identical wells: the wells have \emph{unequal depths} and \emph{unequal spacings}, and are therefore separated by barriers of varying height and width. The quadratic term $x^2/100$ plays two roles: its growth at infinity makes the Gibbs measure well-defined, while its very small coefficient leaves a large number of wells within reach --- ten on $[-30,30]$, sixteen on $[-60,60]$  --- with the outer wells lying progressively higher and hence carrying less mass. This yields a stringent, non-symmetric test: the sampler must discover many minimisers, cross heterogeneous barriers, and recover the correct relative weights of wells of different depths. It also mirrors the molecular-simulation setting of the introduction, where a weakly confined low-energy region contains many metastable states of unequal importance. Moreover, although we know that there is a single free energy minimizer at high  temperature and several localized ones at low temperature, the heterogeneity of the potential is such that we expect different localized mimizers to appear at different critical temperatures. For a given $\sigma^2$, the number of local minimizers is thus not clear at all, and this makes the sampling algorithm particularly useful: contrary to the double-well case considered in Section~\ref{subsec:double-well}, if we wanted to find the minimizers by initializing standard particle systems~\eqref{eq:LangevinN} with various starting distributions, we wouldn't know where to look and  when to stop.

\paragraph{Exploring the wells.}
We simulate the coupled system~\eqref{eq:thetaXcouples} with this $V$ (same convention as~\eqref{eq:dw-coupled}, with $V_1'(x) = -\sin x - \tfrac12\sin\tfrac{x}{2} + \tfrac{x}{50} + x$):
\begin{equation}
\label{eq:mw-coupled}
\dd X_t^i = -\frac{V_1'(X_t^i)}{\sigma^2}\,\dd t + \Theta_t\,\dd t + \sqrt{2}\,\dd B_t^i\,,\qquad
\dd \Theta_t = -\varepsilon\po \sigma^2 \Theta_t - \bar X_t\pf\,\dd t + \sqrt{2\varepsilon/\beta}\,\dd W_t\,.
\end{equation}
Since $\omega$ now has several minima, $\Theta_t$ is metastable and hops between them, and the population $\mu_{\Theta_t}$ correspondingly visits the associated stationary solutions. To identify which solution the particles currently sit in, we use a \emph{fork descent}: at regular intervals we duplicate the particle cloud and let the copy evolve under the $\Theta$-free mean-field Langevin dynamics~\eqref{eq:LangevinN} until its barycentre stabilises; the converged barycentre labels the well whose basin contains the current state. This detector needs no knowledge of $\omega$ and therefore carries over to the non-quadratic case of Section~\ref{subsec:non-quadratic}.

Figure~\ref{fig:mw-explore} shows a run at $\sigma^2 = 0.4$, $\varepsilon = 0.03$, $\beta = 0.05$, $N = 300$, $T = 5000$, with four panels: the potential $V$, the trajectory of $\Theta_t$, the barycentre $\bar X_t$ with the detected minima coloured per well, and the coarse-grained potential $\omega$ computed by quadrature,
\[
\omega(\theta) = \frac{\sigma^2}{2\lambda}\,\theta^2 - \ln \int_{\R} \exp\po -\frac{1}{\sigma^2}\po V(x) + \frac{\lambda}{2}x^2\pf + \theta x \pf \dd x\,.
\]
Over the horizon, $\Theta_t$ drives the population through several wells, and the barycentres returned by the fork descent cluster exactly on the minima of $\omega$ (dashed lines), confirming that the visited states are indeed the stationary solutions of $\mathcal F$.

\begin{figure}[htbp]
    \centering
    \includegraphics[width=\linewidth]{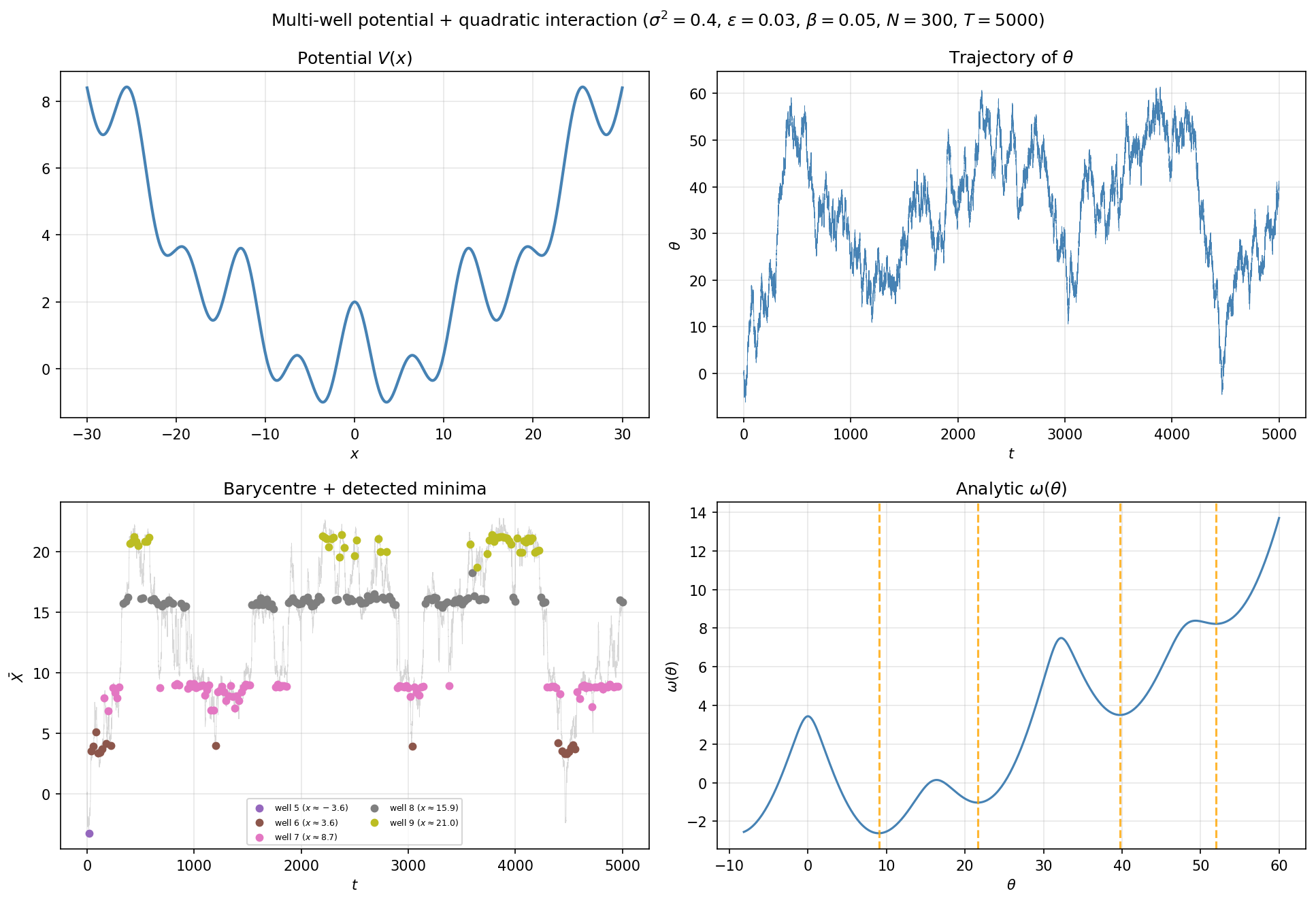}
    \caption{Multi-well potential with quadratic interaction, at $\sigma^2 = 0.4$, $\varepsilon = 0.03$, $\beta = 0.05$, $N = 300$, $T = 5000$. Top left: the potential $V$. Top right: the trajectory of $\Theta_t$. Bottom left: the barycentre $\bar X_t$ (grey) with the minima detected by the fork descent, coloured by well. Bottom right: the analytic coarse-grained potential $\omega(\theta)$, whose minima (dashed) match the detected barycentres.}
    \label{fig:mw-explore}
\end{figure}

\paragraph{Number of wells explored: influence of $\beta$.}
How many distinct wells the sampler reaches over a given horizon is governed by the macroscopic temperature $\beta^{-1}$. Figure~\ref{fig:mw-beta} reports a long run at $\beta = 0.05$ on the wide range $[-60,60]$ (sixteen wells), $T = 20000$, aggregated over $10$ simulations. The left panel shows the cumulative number of distinct wells visited as a function of time (for each simulation), which keeps growing and reaches $10$ to $14$ wells per trajectory; the right panel shows the visit frequency of each well: the central wells are visited most, the outer ones progressively less, and a fraction of the fork descents do not converge in the allotted time (bar \textsf{NC}).

A sweep over $\beta$ confirms the expected trade-off: for large $\beta$ the sampler stays in one or two wells but every descent converges, whereas for very small $\beta$ it touches almost every well but most descents fail to stabilise; an intermediate value $\beta \approx 0.05$--$0.1$ gives the broadest reliable exploration. This sweep is reported in Figure~\ref{fig:mw-beta-sweep}. 

\begin{figure}[htbp]
    \centering
    \includegraphics[width=\linewidth]{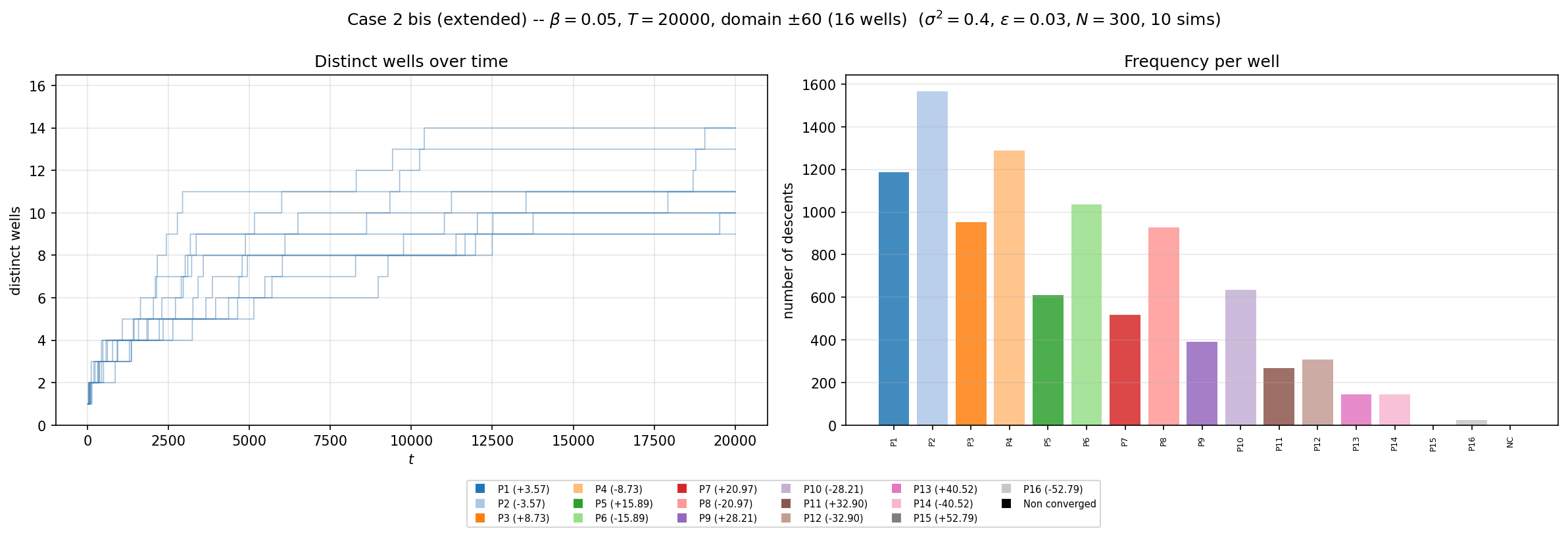}
    \caption{Well exploration at $\beta = 0.05$ on the range $[-60,60]$ (sixteen wells), $\sigma^2 = 0.4$, $\varepsilon = 0.03$, $N = 300$, $T = 20000$, over $10$ simulations. Left: cumulative number of distinct wells visited over time (one step curve per simulation). Right: number of fork descents falling in each well $P_1,\dots,P_{16}$, plus the non-converged ones (\textsf{NC}).}
    \label{fig:mw-beta}
\end{figure}

\begin{figure}[!htbp]
    \centering
    \includegraphics[width=0.92\linewidth]{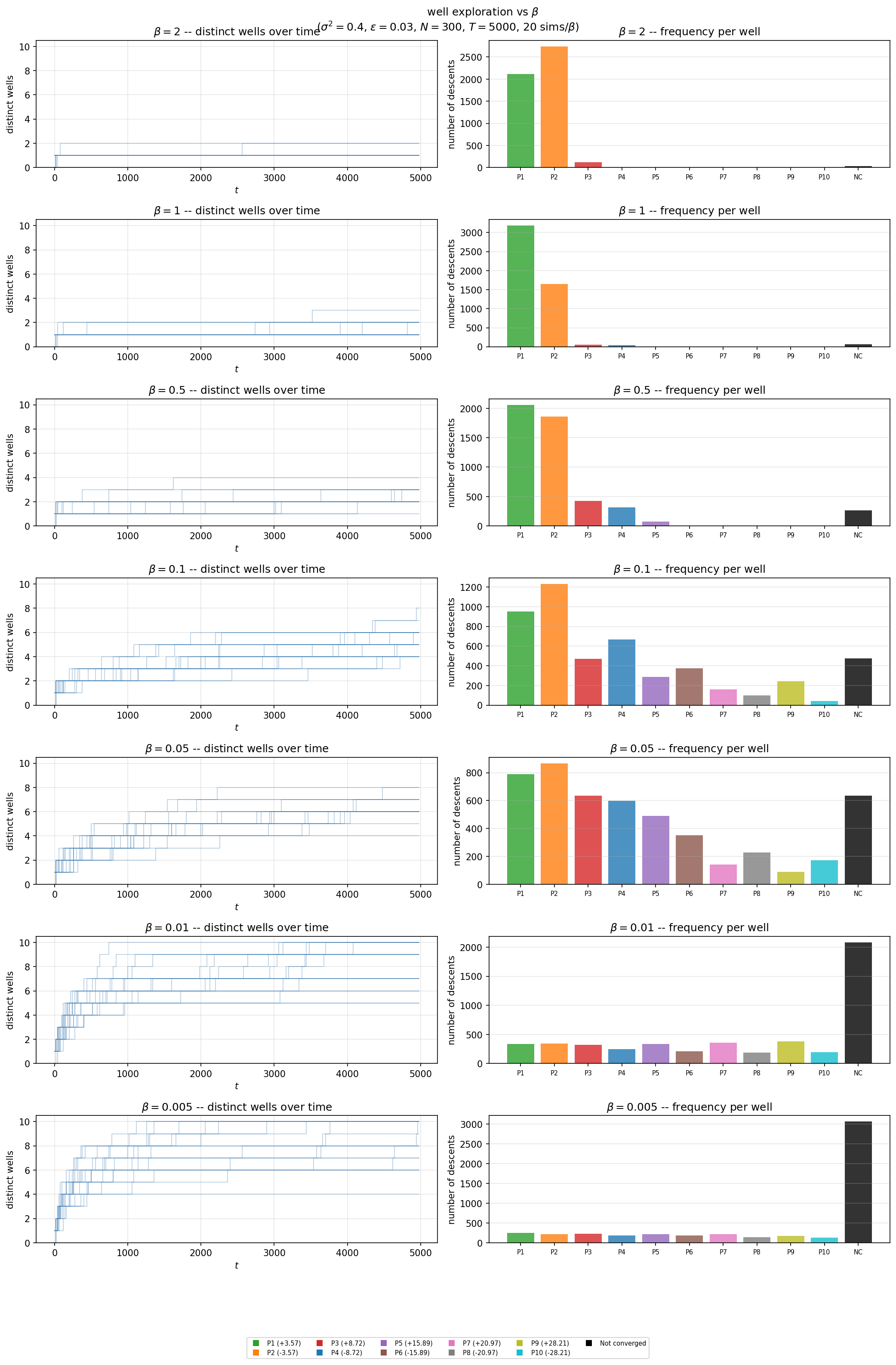}
    \caption{Well exploration as a function of $\beta$ on the range $[-30,30]$ (ten wells), $\sigma^2 = 0.4$, $\varepsilon = 0.03$, $N = 300$, $T = 5000$, with $20$ simulations per $\beta$. From top to bottom: $\beta = 2,\,1,\,0.5,\,0.1,\,0.05,\,0.01,\,0.005$. Left column: cumulative number of distinct wells visited over time (one step curve per simulation). Right column: number of fork descents falling in each well $P_1,\dots,P_{10}$, plus the non-converged ones (\textsf{NC}). As $\beta$ decreases the sampler reaches more wells, but the fraction of non-converged descents grows.}
    \label{fig:mw-beta-sweep}
\end{figure}

\paragraph{Long-time distribution of $\Theta_t$ versus $\nu_\beta$.}
Finally, as in the double-well case, we check that the long-time law of $\Theta_t$ matches the coarse-grained Gibbs measure $\nu_\beta \propto e^{-\beta\omega}$, with $\omega$ evaluated by quadrature. Figure~\ref{fig:mw-nu} shows this comparison at $\sigma^2 = 2$, $\beta = 1$ for three increasing computational budgets. Even the coarsest one ($N = 10^3$, $\varepsilon = 0.1$, $T = 500$) already captures the shape of $\nu_\beta$, and the match becomes very good for the richest one ($N = 10^4$, $M = 20$, $\varepsilon = 0.05$, $T = 10^4$), where the histogram reproduces the multi-modal target closely --- in line with the averaging principle, the approximation improving as $N$ and $M$ increase, the horizon $T$ grows and the speed $\varepsilon$ decreases.

\begin{figure}[!htbp]
    \centering
    \includegraphics[width=0.72\linewidth]{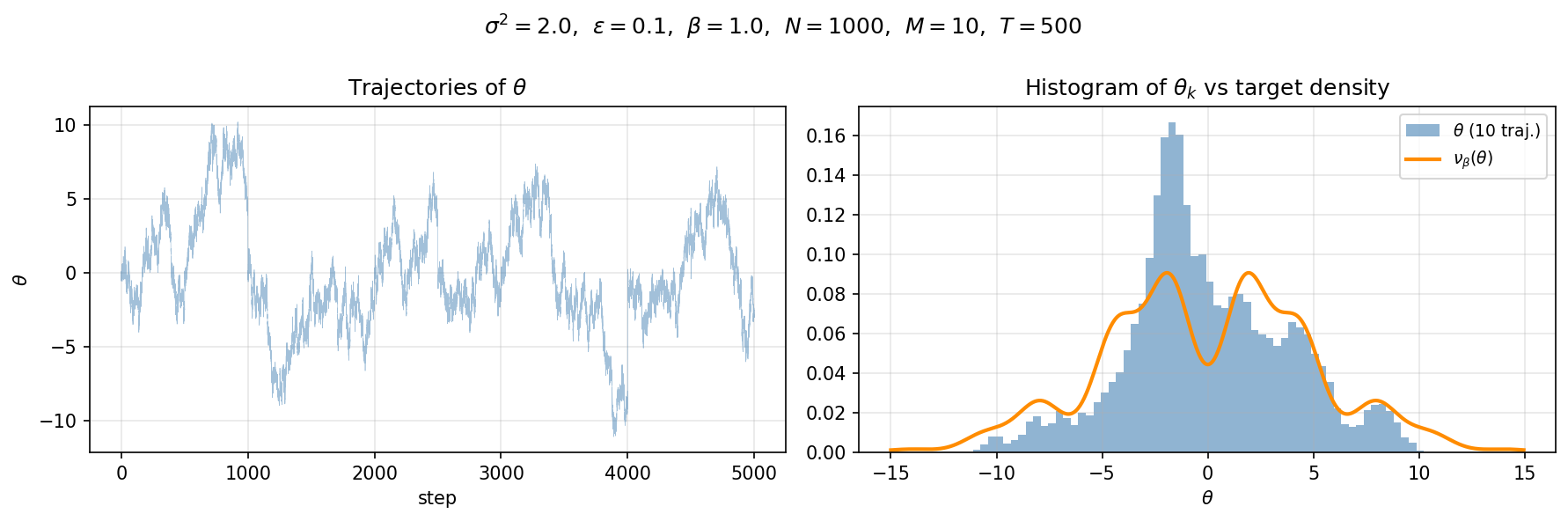}\\[2pt]
    \includegraphics[width=0.72\linewidth]{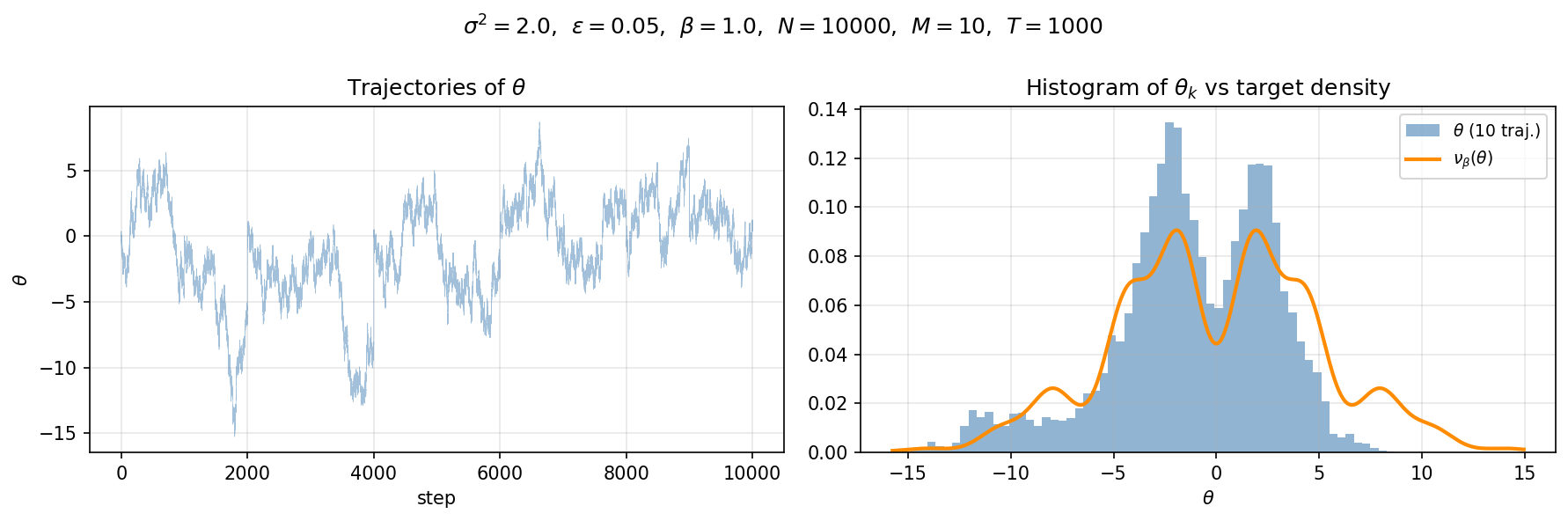}\\[2pt]
    \includegraphics[width=0.72\linewidth]{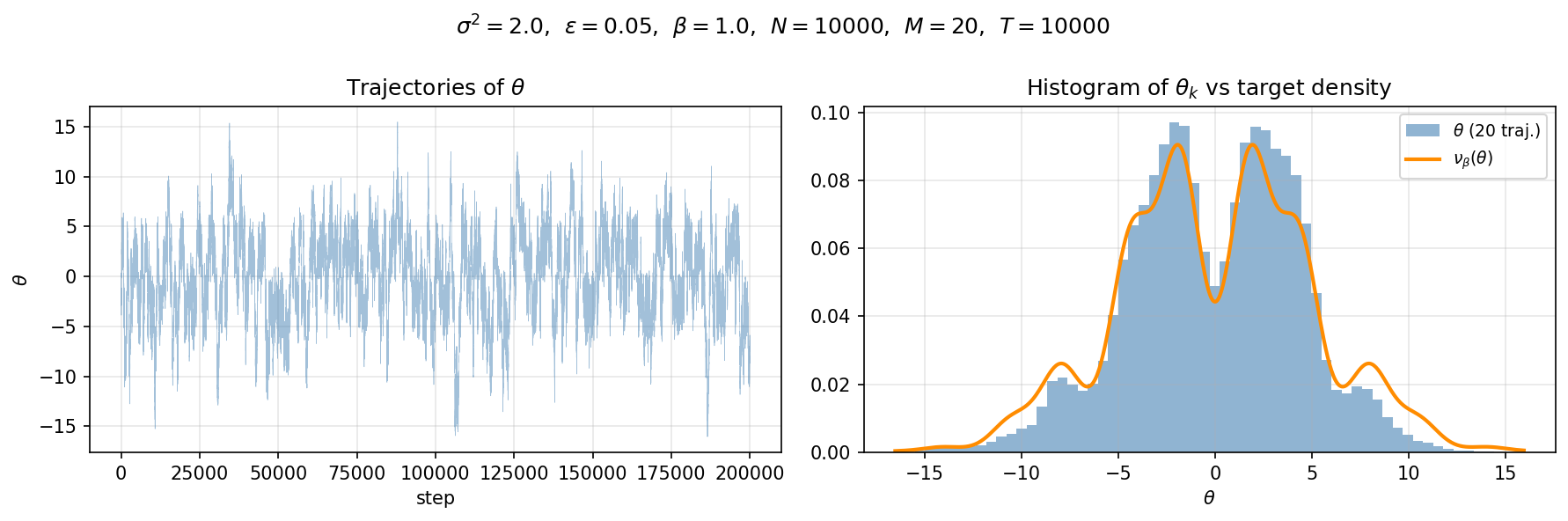}
    \caption{Long-time distribution of $\Theta_t$ for the multi-well potential at $\sigma^2 = 2$, $\beta = 1$, with an increasing computational budget (top to bottom). Top: $N=10^3$, $\varepsilon=0.1$, $M=10$, $T=500$. Middle: $N=10^4$, $\varepsilon=0.05$, $M=10$, $T=10^3$. Bottom: $N=10^4$, $\varepsilon=0.05$, $M=20$, $T=10^4$. Left: concatenated trajectories of $\Theta_t$; right: histogram of $\Theta_t$ against $\nu_\beta \propto e^{-\beta\omega}$. The match improves as $N$, $M$ and $T$ increase and $\varepsilon$ decreases.}
    \label{fig:mw-nu}
\end{figure}

\clearpage

\subsection{Non-quadratic interaction}
\label{subsec:non-quadratic}

We now leave the quadratic setting and consider the attractive Gaussian interaction
\begin{equation}\label{eq:expWtot}
W(x-y) = -\lambda\,\exp\po -\frac{|x-y|^2}{2\eta^2}\pf\,,    
\end{equation}
in dimension $d=1$, keeping the double-well confining potential $V(x) = x^4/4 - x^2/2$. Here $\lambda>0$ sets the strength of the attraction and $\eta>0$ its range. For the plain particle system~\eqref{eq:LangevinN} nothing changes --- one simply uses this $W$ --- and the same kind of phase transition as in Section~\ref{subsec:double-well} occurs, now at a temperature depending on $\lambda$ and $\eta$. Throughout this section, $\lambda = \eta = 1$ and $\sigma^2 = 0.15$.

\paragraph{Reduction to a finite macroscopic variable.}
 The two-scale construction   applies once $W$ is put in the general form~\eqref{eq:geneE}, through the expansion 
\begin{equation}\label{eq:Wexp}
    W(x-y) = -\lambda\,e^{-\frac{x^2+y^2}{2\eta^2}}\sum_{k\geqslant 0} \frac{(xy)^k}{\eta^{2k}k!}
       = -\lambda \sum_{k\geqslant 0} n_k(x)\,n_k(y)\,,\qquad
n_k(x) = e^{-\frac{x^2}{2\eta^2}}\,\frac{x^k}{\eta^k\sqrt{k!}}\,.
\end{equation}
In fact, we will rather implement the two-scale sampler associated to the truncated potential
\[
W_m(x-y) =  
        -\lambda \sum_{k=0}^m n_k(x)\,n_k(y)\,,
\]
for some $m\in\N$. In that case,  we take the feature $\varphi = (n_k)_{k\in\cco 0,m\ccf}:\R\rightarrow\R^{m+1}$. As in Section~\ref{sec:example}, it would have been possible to consider this feature (with a finite $m$) even with the full potential~\eqref{eq:Wexp}, by incorporating the remainder $-\frac{\lambda}{2}\sum_{k>m}\po\int n_k\dd\rho\pf^2 = \frac12 \rho((W-W_m)\star \rho)$ in the term $\mathcal E_0$. Instead, we neglect this remainder. We have two reasons to do this: first, it is computationally cheaper, since computing the pairwise interactions  has a complexity $\mathcal O(N^2)$ when the potential is $W$ (or $W-W_m$) but only $\mathcal O(Nm)$ with $W_m$ (since we only need to compute $ \frac1N \sum_{i=1}^N n_k(X_t^i)$  for $k\in\cco 0,m\ccf$). The second reason is that, in order to assess the properties of the algorithm, we want to compare it to a clear reference. In particular, as in Figures~\ref{fig:dw-nu} and~\ref{fig:mw-nu} we want to compare the distribution sampled by the macroscopic variable with the theoretical macroscopic Gibbs measure $\nu_\beta$. By truncating the potential, similarly to the previous experiments with quadratic interaction, we can compute $\omega$ over $\R^{1+m}$ by a simple quadrature for small values of $m$. As detailed below, the impact of the choice of $m$ is discussed in Figure~\ref{fig:nq-landscape} and, after that, the experiments are conducted with $m=3$ (the macroscopic variable $\Theta_t$ is thus a $4$-dimensional vector).

 Concretely, using the rescaling of~\eqref{eq:rescaled-coupled} and writing $\bar n_{k,t} = \frac1N\sum_{i=1}^N n_k(X_t^i)$ for the empirical feature means, the coupled system reads
\begin{equation}
\label{eq:nq-coupled}
\dd X_t^i = -\frac{V'(X_t^i)}{\sigma^2}\,\dd t + \sum_{k=0}^{m}\Theta_t^k\, n_k'(X_t^i)\,\dd t + \sqrt{2}\,\dd B_t^i\,,\quad
\dd \Theta_t^k = -\varepsilon\po \frac{\sigma^2}{\lambda}\,\Theta_t^k - \bar n_{k,t}\pf\dd t + \sqrt{2\varepsilon/\beta}\,\dd W_t^k\,.
\end{equation}

Nevertheless, when proceeding to fork descents (as explained in Section~\ref{subsec:multi-well}) with the standard Langevin process~\eqref{eq:LangevinN}, we still use the full potential~\eqref{eq:expWtot} (with complexity $\mathcal O(N^2)$ at each step). In other words, the truncation from $W$ to $W_m$ is another approximation (in top of the approximation in $N$ and $\varepsilon$), but we still use the resulting process as an algorithm to sample good initial conditions for a gradient descent converging to the local minimizers of the free energy corresponding to the Gaussian interaction~\eqref{eq:expWtot}.


\paragraph{Coarse-grained potential and stationary solutions.}

The stationary solutions of $\mathcal F$ are obtained by a fixed-point iteration on the self-consistency equation $\rho\propto e^{-(V+W\ast\rho)/\sigma^2}$.  As in the previous sections, for $m\in\N$, the coarse-grained potential is computed by quadrature,
\[
\omega(\theta) = \frac{\sigma^2}{2\lambda}\,|\theta|^2 - \ln \int_{\R} \exp\po -\frac{V(x)}{\sigma^2} + \theta\cdot \varphi(x) \pf \dd x\,,\qquad \mu_\theta \propto \exp\po -\frac{V}{\sigma^2} + \theta\cdot\varphi\pf\,.
\]Figure~\ref{fig:nq-landscape}(a) shows the phase transition: for the exact interaction~\eqref{eq:expWtot}, $\mathcal F$ has two minimisers $\mu_\pm$, with barycentres $\pm m_*$, below $\sigma_c^2\approx 0.275$, and a single symmetric one above; the truncated models (same $\mathcal F$ except that $W$ is replaced by $W_m$) reproduce this bifurcation, and converge to the exact one as $m$ increases. At $\sigma^2 = 0.15$, $m_*\approx 0.86$. As discussed in Section~\ref{sec:intro_twoscale}, the local minimisers of $\omega$ correspond to those of $\mathcal F$: Figure~\ref{fig:nq-landscape}(b) shows that $\mu_{\theta_*}$, at the two minima $\theta_*^\pm$ of $\omega$ (computed with $m=3$), coincides with the minimisers $\mu_\pm$ (up to an error $\mathcal W_2\approx 0.03$ for $m = 3$, a distance which decreases geometrically in $m$, see the inset plot), and Figure~\ref{fig:nq-landscape}(c) shows the profile of $\omega$ (for $m=3$) along $\theta^1$, the component conjugate to the odd feature $n_1$ that distinguishes the two wells: a symmetric double well with barrier $\Delta\omega\approx 0.29$, whose minima are indistinguishable in the graph from the values $\sigma^{-2}\mathcal F(\mu_\pm)$ with the full potential $W$ (as determined by the fixed-point procedure).

\begin{figure}[htbp]
    \centering
    \includegraphics[width=\linewidth]{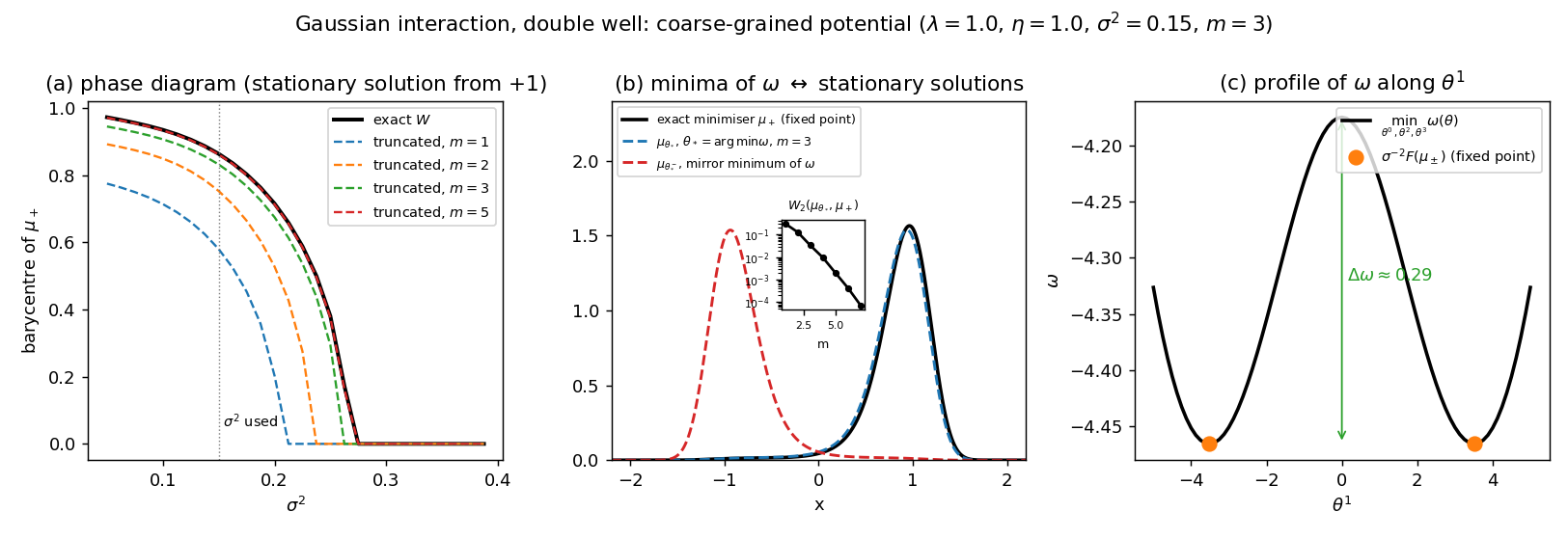}
    \caption{Gaussian interaction, double-well potential, $\lambda = \eta = 1$, $\sigma^2 = 0.15$, computed on a grid. (a) Barycentre of the stationary solution obtained from $+1$ versus $\sigma^2$, for the exact interaction and for the truncations at order $m$. (b) Exact minimiser $\mu_+$ versus $\mu_{\theta_*^\pm}$ at the two minima of $\omega$, $m = 3$; inset: $\mathcal W_2(\mu_{\theta_*},\mu_+)$ versus $m$. (c) Profile $\theta^1\mapsto\min_{\theta^0,\theta^2,\theta^3}\omega(\theta)$, with the values $\sigma^{-2}\mathcal F(\mu_\pm)$ of the fixed-point minimisers (dots).}
    \label{fig:nq-landscape}
\end{figure}

\paragraph{Exploring the wells in the vector $\Theta$.}
Figure~\ref{fig:nq-explore} shows a run of~\eqref{eq:nq-coupled} at $\sigma^2 = 0.15$, $\lambda = \eta = 1$, $m = 3$, $\varepsilon = 0.3$, $\beta = 4$, $N = 300$, $T = 3000$. As in Section~\ref{subsec:multi-well}, every $20$ time units we launch a fork descent, here with the exact Gaussian interaction~\eqref{eq:expWtot}. The component $\Theta^1_t$ hops between the two minima $\pm\theta_*^1$ of $\omega$, and the well detected by the fork descent follows its sign: the vector $\Theta_t$ drives the transitions between the two wells, exactly as the scalar $\Theta_t$ did in the quadratic case. All the descents converge to the stationary solutions $\mu_\pm$ (their barycentres are $\pm 0.85$, against $\pm m_*$). Along the run, the particles are distributed according to $\mu_{\Theta_t}$, as stated in Theorem~\ref{thm:averaging}: when $\Theta_t$ is in a well, $\mathcal W_2(\pi_{\bX_t},\mu_{\Theta_t})$ is at the level of the distance between $N$ i.i.d.\ samples of $\mu_{\Theta_t}$ and $\mu_{\Theta_t}$ itself (medians $0.045$ and $0.040$), with a slight lag during the transitions (medians $0.072$ and $0.051$).

\begin{figure}[htbp]
    \centering
    \includegraphics[width=\linewidth]{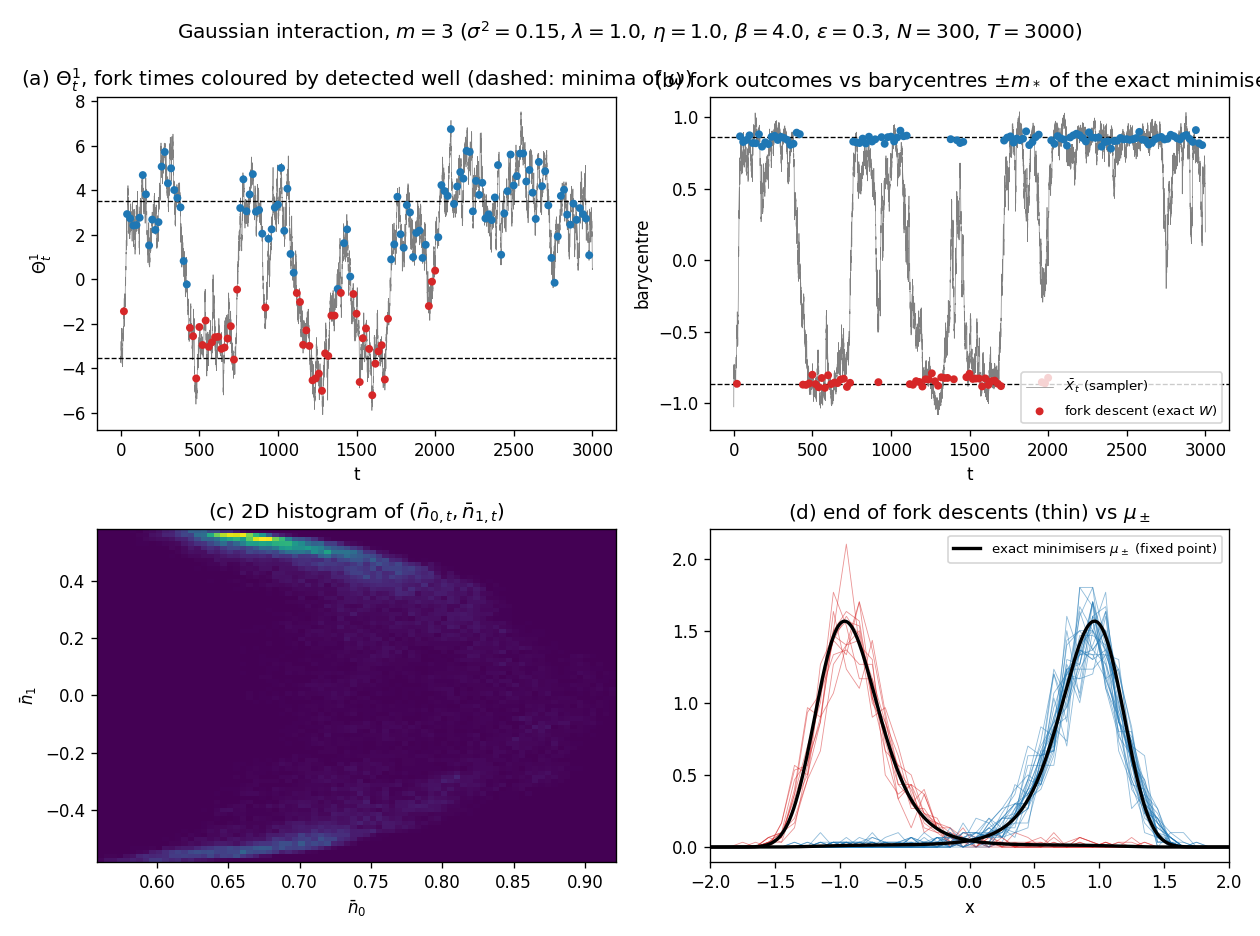}
    \caption{Gaussian interaction, double-well potential, at $\sigma^2 = 0.15$, $\lambda = \eta = 1$, $m = 3$, $\varepsilon = 0.3$, $\beta = 4$, $N = 300$, $T = 3000$. (a) Trajectory of $\Theta^1_t$, with the fork times coloured by the well detected during the descent (dashed: $\pm\theta^1_*$). (b) Barycentre $\bar X_t$ (grey) and barycentres reached by the fork descents with the exact Gaussian interaction (dots); dashed: $\pm m_*$. (c) Two-dimensional histogram of the feature means $(\bar n_{0,t},\bar n_{1,t})$. (d) Final densities of the fork descents versus the stationary solutions $\mu_\pm$.}
    \label{fig:nq-explore}
\end{figure}

\paragraph{Transition times and number of particles.}
For the plain particle system~\eqref{eq:LangevinN}, the equilibrium of the macroscopic variable is  $\nu_N\propto e^{-N\omega_N}$ (see~\eqref{eq:omegaN}) and  transitions between $\mu_+$ and $\mu_-$ are rare events whose rate decays like $e^{-N\Delta\omega}$, while for the two-scale sampler $\Theta_t$ is distributed according to $\nu_\beta\propto e^{-\beta\omega}$ and the transition rate does not depend on $N$. Figure~\ref{fig:nq-exit} compares both, counting the switches of the barycentre between the two wells (with the truncated interaction for the plain system, so that both target the same free energy). For the plain system the rate decays exponentially, from $7\cdot10^{-2}$ at $N = 4$ to $10^{-4}$ at $N = 28$, with a slope $-0.28$ close to $\Delta\omega$; for the sampler it stays between $2.6\cdot10^{-3}$ and $5.3\cdot10^{-3}$ from $N = 10$ to $N = 1000$.

\begin{figure}[htbp]
    \centering
    \includegraphics[width=\linewidth]{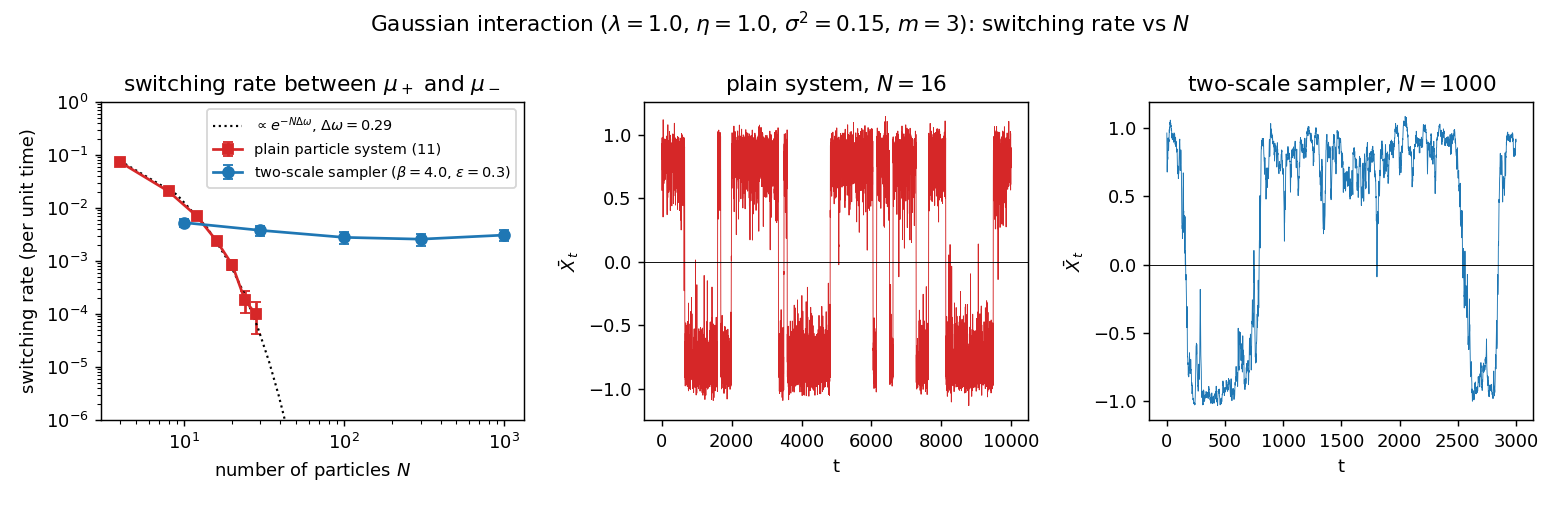}
    \caption{Gaussian interaction, $\sigma^2 = 0.15$, $m = 3$. Left: switching rate between $\mu_+$ and $\mu_-$ versus $N$, for the plain particle system (squares, $12$ runs of length $T = 10^4$ per $N$) and for the two-scale sampler with $\beta = 4$, $\varepsilon = 0.3$ (circles, $8$ runs of length $T = 3000$); error bars $\pm2\sqrt{\#\text{switches}}$; dotted: $e^{-N\Delta\omega}$. Middle and right: barycentre trajectories of the plain system ($N = 16$) and of the sampler ($N = 1000$).}
    \label{fig:nq-exit}
\end{figure}

\paragraph{Long-time distribution of $\Theta_t$ versus the target $\nu_\beta$.}
Finally, as in Sections~\ref{subsec:double-well} and~\ref{subsec:multi-well}, we check that the long-time law of $\Theta_t$ is the coarse-grained Gibbs measure $\nu_\beta\propto e^{-\beta\omega}$ predicted by the averaging principle. Since $\Theta_t\in\R^4$, we compare marginals: as before, $\nu_\beta$ is computed by quadrature, its marginals being obtained by integrating $e^{-\beta\omega}$ over the other components of $\theta$. Figure~\ref{fig:nq-nu} does this at $\beta = 4$, $N = 300$, for $\varepsilon \in\{ 1,\ 0.3,\ 0.1\}$ (horizon $T = 500/\varepsilon$, $8$ trajectories started in the left well). Within each well, the law of $\Theta^1_t$ matches that of $\nu_\beta$ for all three values of $\varepsilon$ (total variation distance $0.02$--$0.04$ after symmetrisation $\theta^1\to-\theta^1$). The relative weight of the two wells, on the other hand, is estimated from a limited number of transitions ($15$--$20$ in total for each value of $\varepsilon$), and the fraction of time with $\Theta^1_t>0$ ($0.25$, $0.27$ and $0.45$) fluctuates around its target $1/2$; it improves with the number of trajectories or the horizon.

\begin{figure}[htbp]
    \centering
    \includegraphics[width=\linewidth]{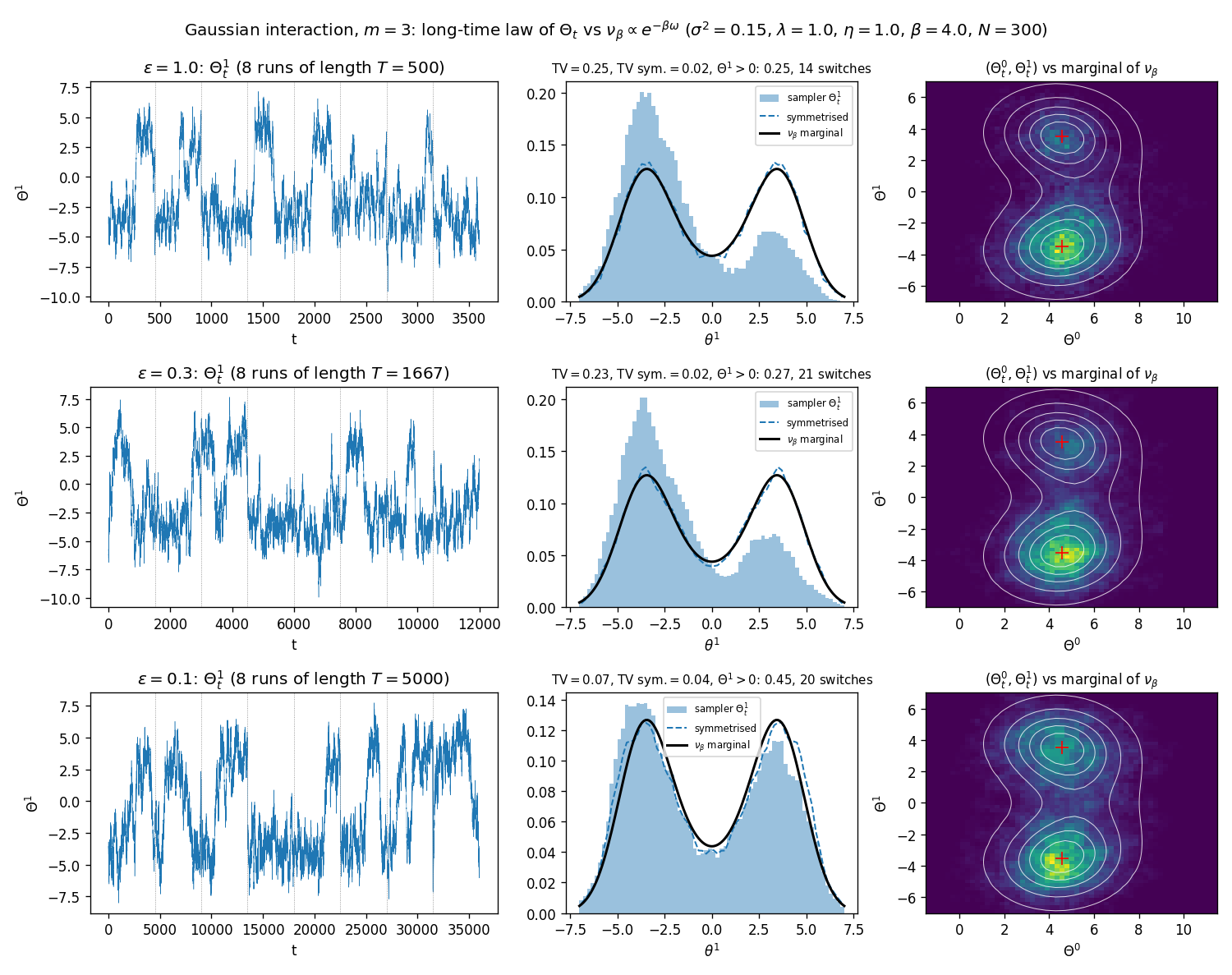}
    \caption{Long-time distribution of $\Theta_t$ at $\sigma^2 = 0.15$, $m = 3$, $\beta = 4$, $N = 300$, aggregated over $8$ trajectories. From top to bottom: $\varepsilon = 1,\ 0.3,\ 0.1$, with $T = 500/\varepsilon$. Left: the trajectories of $\Theta^1_t$, concatenated one after another. Middle: histogram of $\Theta^1_t$ (and its symmetrisation, dashed) against the marginal of $\nu_\beta\propto e^{-\beta\omega}$ computed by quadrature (black). Right: histogram of $(\Theta^0_t,\Theta^1_t)$ with contours of the $(\theta^0,\theta^1)$-marginal of $\nu_\beta$ and the minima of $\omega$ (crosses).}
    \label{fig:nq-nu}
\end{figure}

\clearpage

\section{Proofs}\label{sec:proof}

\subsection{Propagation of chaos}

This section is devoted to the proof of Theorem~\ref{thm:chaos}.

\begin{proposition}\label{prop:well-posed}
In the settings of Theorem~\ref{thm:chaos}, Equations~\eqref{eq:thetaXcouples} and~\eqref{eq:thetamucouples} have unique strong solutions. Moreover, for any $T>0$, $\sup_{t\in[0,T]}\int_{\R^d} |x|^p \nu_t(\dd x)$ is almost surely finite.
\end{proposition}
\begin{proof}
Since $\na_{x_i} V_N(\bx) = D\mathcal E_0(\pi_{\bX},x_i)$,~\eqref{eq:onesidedLip} implies that 
\begin{equation*}
\forall \bx,\by\in \R^{dN},\qquad (\bx-\by) \cdot \po \na V_N(\bx)  - \na V_N(\by)\pf \leqslant L | \bx-\by|^2\,.
\end{equation*}
Well-posedness for Equation~\eqref{eq:thetaXcouples} then follows from standard results for stochastic differential equations. Similarly, well-posedness for~\eqref{eq:thetamucouples}  follows from a standard fixed-point argument. We only detail the $L^p$ moment for $\nu_t$. Writing $f_p(x)=|x|^p$,
\begin{align*}
\dd \int_{\R^d} f_p \dd \nu_t &= \int_{\R^d} \po \na f_p \cdot \po - D\mathcal E_0(\nu_t,\cdot) + \bar \Theta_t \cdot \na \varphi  \pf +\Delta f_p\pf \dd \nu_t \dd t\\
&\leqslant C\po 1 + |\bar \Theta_t|^p + \int_{\R^d} f_p \dd \nu_t\pf\dd t\,,
\end{align*}
for some $C>0$, where we used Young inequality to bound $|\na f_p||\Theta_t| \leqslant C(f_p+|\Theta_t|^p)$, and we used~\eqref{eq:onesidedLip} to get that
\begin{align*}
x \cdot D\mathcal E_0(\nu_t,x) & \leqslant L|x|^2 + x \cdot D\mathcal E_0(\nu_t,0)\\
& \leqslant L|x|^2 + |x||D\mathcal E_0(\delta_0,0)| + L|x| \mathcal W_2(\nu_t,\delta_0) \\
& = L|x|^2 + |x||D\mathcal E_0(\delta_0,0)| + L|x| \sqrt{\int_{\R^d} f_2\dd \nu_t} \\
& \leqslant L|x|^2 + |x||D\mathcal E_0(\delta_0,0)| + L|x| \po \int_{\R^d} f_p\dd \nu_t\pf^{1/p}
\end{align*}
(and then again Young inequality). Grönwall lemma then shows that, for any $T>0$, there exists $C_T>0$ such that, almost surely,
\[\sup_{t\in[0,T]} \int_{\R^d} f_p \dd \nu_t \leqslant C_T\po  1+ \int_{\R^d} f_p \dd \nu_0 + \int_0^T |\Theta_s|^p \dd s\pf\,, \]
which concludes the proof.
\end{proof}

\begin{proof}[Proof of Theorem~\ref{thm:chaos}]

Given a solution of~\eqref{eq:thetamucouples}, denote by $\bY_t=(Y_t^1,\dots,Y_t^N)$  the solution of
\begin{equation}
\label{eq:thetaYMcKeanV}
\forall i\in\cco 1,N\ccf,\qquad \dd Y_t^i = - D\mathcal E_0(\nu_t,Y_t^i) \dd t + \bar\Theta_t \cdot \na \varphi(Y_t^i) \dd t + \sqrt{2}\dd B_t^i\,,
\end{equation}
with initial conditions $\bY_0=\bX_0$. Write $(\mathcal F_t)_{t\geqslant 0}$ the filtration associated to $(B_t)_{t\geqslant 0}$. Conditionally to $\mathcal F_t$, $Y_t^1,\dots,Y_t^N$ are independent and distributed according to $\nu_t$. In particular, for $i\in\cco 1,N\ccf$, using that $\na_{x_i} V_N(\bx) = D\mathcal E_0(\pi_{\bX},x_i)$,
\[\mathbb E \po |D\mathcal E_0(\nu_t,Y_t^i) - \na_{x_i} V_N(\bY_t)|^2 \ |\ \mathcal F_t\pf \leqslant L^2 \mathbb E \po \mathcal W_2^2(\pi_{\bY_t},\nu_t) \ |\ \mathcal F_t\pf  \leqslant C  N^{-\eta} \po \int_{\R^d} |y|^p \nu_t(\dd y) \pf^{2/p}\,, \]
for some constant $C,\eta>0$  thanks to \cite{FournierGuillin} and Proposition~\ref{prop:well-posed}. 
 Then, since $X^i$ and $Y^i$ are driven by the same Brownian motion $B^i$,
\begin{align*}
\dd |\bX_t- \bY_t|^2 & \leqslant 2(\bX_t - \bY_t)\cdot \po \na V_N(\bX_t) - \na V_N(\bY_t)\pf + 2 \sum_{i=1}^N (X_t^i - Y_t^i) \\
& \quad \times \po   |D\mathcal E_0(\nu_t,Y_t^i) - \na_{x_i} V_N(\bY_t)|  + |\bar\Theta_t||\na \varphi(X_t^i) - \na \varphi(Y_t^i) | + |\Theta_t-\bar\Theta_t| |\na \varphi(X_t^i)|\pf \dd t \\
&\leqslant C (1+|\bar\Theta_t|) \po  |\bX_t-\bY_t|^2 + N |\Theta_t-\bar \Theta_t|^2 + \sum_{i=1}^N  |D\mathcal E_0(\nu_t,Y_t^i) - \na_{x_i} V_N(\bY_t)|^2 \pf\dd t\,,
\end{align*}
for some $C>0$. Similarly,
\begin{align*}
\dd |\bar\Theta_t - \Theta_t|^2 & \leqslant 2\varepsilon|\bar\Theta_t-\Theta_t|\po |\bar\Theta_t-\Theta_t| + L \mathcal W_2(\nu_t,\pi_{\bX_t}) \pf\dd t \\
& \leqslant C \po |\bar\Theta_t-\Theta_t|^2 + \frac1N |\bX_t-\bY_t|^2 + \mathcal W_2^2(\nu_t,\pi_{\bY_t}) \pf \dd t 
\end{align*}
for some $C>0$. Writing 
\[R_t = \mathbb E \po\left. |\bar\Theta_t-\Theta_t|^2 + \frac1N|\bX_t-\bY_t|^2\ \right|\ \mathcal F_t \pf\,, \]
combining the previous bounds (and relying again on \cite{FournierGuillin}) gives
\[\dd R_t \leqslant H_t \po R_t + N^{-\eta}\pf \dd t \]
for some $\mathcal F_t$-measurable random $H_t$, almost surely bounded on any finite time interval, and independent from $N$. By Grönwall lemma, since $R_0=0$, we get that
\[ R_t \leqslant H_t' N^{-\eta}\]
for some $\mathcal F_t$-measurable random  $H_t' <\infty$ (non-decreasing with $t$). Moreover, as we used above,
\[\mathcal W_2^2(\pi_{\bX_t}, \nu_t) \leqslant \frac{2}{N}|\bX_t-\bY_t|^2 + 2\mathcal W_2^2(\nu_t,\pi_{\bY_t})\,,\]
which combined with the previous bound and \cite{FournierGuillin} leads to
\begin{equation}
\label{loc:sdvdlfmlkqsd}
  \mathbb E \po \mathcal W_2^2(\pi_{\bX_t}, \nu_t)  + |\Theta_t - \bar\Theta_t|^2\ |\ \mathcal F_t\pf \leqslant S_t N^{-\eta}\,,
\end{equation}
for an $\mathcal F_t$-measurable   random   $S_t<\infty $ (non-decreasing with $t$). Then, for an arbitrary $K>0$, we bound, for any $t\in [0,T]$ and $\delta>0$,
\begin{align*}
 \mathbb P \po \mathcal W_2(\pi_{\bX_t}, \nu_t)  + |\Theta_t - \bar\Theta_t| \geqslant \delta \pf 
& \leqslant  \mathbb P \po \mathcal W_2(\pi_{\bX_t}, \nu_t)  + |\Theta_t - \bar\Theta_t| \geqslant \delta, \ S_t \leqslant K \pf + \mathbb P \po S_t >K \pf\\
&\leqslant \delta^{-2}  KN^{-\eta} +   \mathbb P\po  S_t> K \pf \,.
\end{align*}
Taking the limsup as $N\rightarrow \infty$ and then letting $K\rightarrow \infty$ concludes the proof.

\end{proof}

\begin{remark}\label{rem:a.s.CV}
From~\eqref{loc:sdvdlfmlkqsd}, we get that, for all $T>0$, almost surely,
\[\sup_{t\in[0,T]} \sup_{N\geqslant 1} N^{\eta} \mathbb E \po \mathcal W_2^2(\pi_{\bX_t}, \nu_t)  + |\Theta_t - \bar\Theta_t|^2\ |\ \mathcal F_t\pf < \infty\,,\]
where $\eta$ is the rate provided by \cite{FournierGuillin} (which depends explicitly on $p$ and $d$; for instance, if $d>4$ and $p=4$, then $\eta=2/d$).
\end{remark}

\subsection{Averaging}

This section is devoted to the proof of Theorem~\ref{thm:averaging}.

\subsubsection{PL inequalities and consequences}

We start with some properties of the free energies $\mathcal F_\theta$, $\theta\in\R^m$.

\begin{proof}[Proof of Proposition~\ref{prop:lambda_thetaPL}]
From the lower-bound~\eqref{eq:lowerboundE0m2}, for any $\theta\in\R^m$,
\begin{align}
\mathcal F_{\theta}(\nu) &\geqslant c_0 \int_{\R^d} |x|^2 \nu(\dd x) - C_0 - |\theta| \po |\varphi_\ell(0)|  + \|\na \varphi_\ell\|_\infty \int_{\R^d} |x|\nu(\dd x)  + \|\varphi_b\|_\infty\pf + \mathcal H(\nu)  \nonumber \\
& \geqslant \frac{c_0}2 \int_{\R^d} |x|^2 \nu(\dd x) - C_0 - |\theta| \po |\varphi_\ell(0)|  +  \|\varphi_b\|_\infty\pf - \frac{1}{2c_0} |\theta|^2 \|\na \varphi_\ell\|_\infty ^2 + \mathcal H(\nu) \nonumber \\
& = \frac{c_0}4 \int_{\R^d} |x|^2 \nu(\dd x) +  \frac{c_0}{4} \mathcal H\po \nu|\gamma\pf - C_1(1+|\theta|^2)\label{loc:sfgkdflsdgz}
\end{align}
for some constant $C_1$, where $\gamma = \mathcal N(0,(2c_0)^{-1} I)$. As a consequence, $\mathcal F_{\theta}$ is lower-bounded, and $\mathcal F_\theta(\nu)$ goes to infinity as $\int_{\R^d} |x|^2 \nu(\dd x) \rightarrow \infty$. By semi-lower continuity, $\mathcal F_\theta$ admits a global minimiser $\mu_\theta$.

 Since the term $\nu \mapsto \theta\cdot \nu(\varphi)$ is linear, 
 \[\mathcal E_\theta(\nu) := \mathcal E_0(\nu) - \theta\cdot \nu(\varphi)\]
 satisfies the same condition~\eqref{eq:convexity-c} as $\mathcal E_0$.
From  \cite[Lemma 3]{monmarche2025local}, for all $\nu\in\mathcal P_2(\R^d)$,
\begin{equation}
\label{loc:dgkjdfgdfg}
\tF_\theta(\nu)  \leqslant \mathcal H\po \nu|\Gamma_\theta(\nu)\pf + \eta\|\nu - \mu_\theta\|_{TV}^2
\end{equation}
 with the probability density
 \[\Gamma_\theta(\nu) \propto \exp \po - \frac{\delta\mathcal E_\theta}{\delta m}(\nu,\cdot)  \pf = \exp \po - \frac{\delta\mathcal E_0}{\delta m}(\nu,\cdot) + \theta \cdot \varphi \pf\,,\]
 and
 \[\mathcal H(\nu|\mu_\theta) \leqslant \tF(\nu) + \eta\|\nu - \mu_\theta\|^2_{TV}\,.\]
 Combining  Pinsker's inequality and the previous ones gives
 \[\|\nu-\mu_\theta\|^2_{TV} \leqslant 2 \mathcal H(\nu|\mu_\theta) \leqslant 2   \mathcal H\po \nu|\Gamma_\theta(\nu)\pf + 4\eta \|\nu-\mu_\theta\|_{TV}^2\,,\]
 so that
 \begin{equation}\label{loc:sdgdfhmlmq}
  \|\nu-\mu_\theta\|^2_{TV} \leqslant \frac{2}{1-4\eta }   \mathcal H\po \nu|\Gamma_\theta(\nu)\pf \,. 
 \end{equation}
Applying this to $\nu$ being a critical point of $\mathcal F_\theta$ (i.e. a solution of $\nu = \Gamma_{\theta}(\nu)$) shows that $\nu = \mu_\theta$. In particular, $\mu_\theta$ is the unique global minimiser of $\mathcal F_\theta$. Besides, combining~\eqref{loc:dgkjdfgdfg} and~\eqref{loc:sdgdfhmlmq} gives
 \begin{equation}
 \label{loc:fùghfg}
 \forall \nu\in\mathcal P_2(\R^d),\qquad  \tF_\theta(\nu) \leqslant \po 1 + \frac{2\eta}{1-4\eta}\pf \mathcal H(\nu|\Gamma_\theta(\nu))\,. 
 \end{equation}

 Under Assumption~\ref{assum:PL}, for any $\theta=(\theta_\ell,\theta_b)\in \R^{m_\ell}\times\R^{m_b}$ and $\nu\in\mathcal P_2(\R^d)$, $\frac{\delta\mathcal E_\theta}{\delta m}(\nu,\cdot)$ is the sum of three functions: a $\kappa$-strongly convex one (including the term $\theta_\ell\cdot \varphi_\ell$) , a $M+|\theta_b|\|\varphi_b\|_\infty$-bounded one and a $L$-Lipschitz one. By the Bakry-Emery condition and the stability of log-Sobolev inequalities with respect to bounded and Lipschitz perturbations (see \cite{BakryGentilLedoux}), for any $\nu \in\mathcal P_2(\R^d)$, $\Gamma_\theta(\nu)$ satisfies 
 \[\forall \nu,\mu \in \mathcal P_2(\R^d),\qquad \mathcal H(\mu|\Gamma_\theta(\nu)) \leqslant \frac{e^{|\theta_b|\|\varphi_b\|_\infty}}{\lambda'} \mathcal I\po \mu|\Gamma_\theta(\nu)\pf\,,\]
 for some $\lambda'>0$ depending only on $\kappa,M$ and $L$. Applying this inequality with $\mu = \nu$, plugging it in~\eqref{loc:fùghfg} and noticing that
 \[\mathcal I\po \nu|\Gamma_\theta(\nu)\pf = \int_{\R^d} \left|\na \ln \nu + \frac{\delta\mathcal E_\theta}{\delta m}(\nu,\cdot) \right|^2 \nu = \mathcal I_\theta(\nu) \]
concludes the proof.
\end{proof}

\begin{lemma}\label{lem:muthetaLip}
In the settings of  Proposition~\ref{prop:lambda_thetaPL}, assuming that $\varphi$ is $L_\varphi$-Lipschitz continuous, 
\begin{equation}
\label{eq:muthetaLipschitz}
\forall \theta,\psi \in \R^m,\qquad \mathcal W_2(\mu_\theta,\mu_\psi) \leqslant \frac{3L_\varphi}{\lambda_\psi} |\psi-\theta|\,.
\end{equation}
\end{lemma}

\begin{proof}
For a fixed $\theta\in\R^m$, along the flow~\eqref{eq:nu_theta_fixed}, by a classical computation (see \cite{ambrosio2005gradient}), for all $t>0$,
\[\partial_t \tF_{\theta}(\rho_t) = - \mathcal I_{\theta}(\rho_t)\,.\]
Keeping the same flow with this $\theta$ and considering another $\psi \in \R^m$,
\[\partial_t \tF_{\psi}(\rho_t) =  - \mathcal I_{\theta}(\rho_{t}) + (\theta-\psi)\cdot  \partial_t \rho_{t}(\varphi) \,.\]
Integrating by parts,
\[|\partial_t\rho_t(\varphi)| = \left|\int_{\R^d} \na \varphi\cdot  \na \co \ln \rho_t + D\mathcal E_0(\rho_t,\cdot) - \theta\cdot \na \varphi \cf \rho_t\right| \leqslant L_\varphi \sqrt{\mathcal I_{\theta}(\rho_t)} \]
and thus
\begin{align*}
\partial_t \tF_{\psi}(\rho_t)
& \leqslant - \frac12 \mathcal I_{\theta}(\nu_{t}) +  \frac12  L_\varphi^2 |\theta-\psi|^2  \\
& \leqslant - \frac14 \mathcal I_{\psi}(\nu_{t}) + 2 L_\varphi^2 |\theta-\psi|^2  \\
&\leqslant - \frac{\lambda_\psi}2  \widetilde{\mathcal F}_{\psi}(\rho_{t}) + 2 L_\varphi^2 |\theta-\psi|^2  \,,
\end{align*}
using the  PL inequality~\eqref{eq:PLFtheta}. Applying this to $\rho_0 = \mu_\theta$ (so that $\rho_t = \mu_\theta$ for all $t\geqslant 0$) gives
\begin{equation}
\label{loc:dfgdfgk}
\widetilde{\mathcal F}_\psi(\mu_{\theta}) \leqslant \frac{4L_\varphi^2}{\lambda_\psi} |\theta-\psi|^2  \,.
\end{equation}
Moreover, the PL inequality is known to imply the transport inequality
\begin{equation}
\label{eq:Talagrand}
\forall \nu\in\mathcal P_2(\R^d)\,,\qquad \mathcal W_2^2(\nu,\mu_\psi) \leqslant \frac{2}{\lambda_\psi} \widetilde{\mathcal F}_\psi(\nu)\,,
\end{equation}
cf. \cite{Pavliotis,CMCV,monmarche2025local}. Combining the two last inequalities establishes~\eqref{eq:muthetaLipschitz}. 
\end{proof}

\begin{lemma}\label{lem:Ftildepsitheta}
Under the settings of Lemma~\ref{lem:muthetaLip},  for all $\theta,\psi\in\R^m$ and $\rho\in\mathcal P_2(\R^d)$, 
\[\tF_{\theta}(\rho) \leqslant  2 \tF_{\psi} (\rho) + \frac{4L_\varphi^2}{\lambda_\psi} |\psi-\theta|^2\,.\]
\end{lemma}
\begin{proof}
Using that $\mu_\theta$ minimises $\mathcal F_\theta$,
\begin{align*}
\tF_\theta(\rho) &= \mathcal F_\theta(\rho) - \mathcal F_\theta(\mu_\theta) \\
& = \mathcal F_\psi(\rho) + (\psi-\theta)\cdot\po \rho(\varphi) - \rho(\mu_\theta)\pf   -  \mathcal F_\psi(\mu_\theta) \\
& \leqslant \tF_\psi(\rho) + L_\varphi|\psi-\theta| \po \mathcal W_2(\rho,\mu_\psi) + \mathcal W_2(\mu_\theta,\mu_\psi)\pf   \\
& \leqslant \tF_\psi(\rho) + \frac{\lambda_\psi}{2} \mathcal W_2^2(\rho,\mu_\psi) +   \frac{4L_\varphi^2}{\lambda_\psi}  |\psi-\theta|^2 \,.
\end{align*}
where we used Lemma~\ref{lem:muthetaLip}. The transport inequality~\eqref{eq:Talagrand} concludes the proof.
\end{proof}

\subsubsection{Truncation}

Since the PL constant $\lambda_\theta$ provided in~\eqref{eq:lambdathetapropr2} is not bounded from below uniformly  in $\theta$, we need to consider a truncated process for which this constant is uniform in time. We will then prove the averaging principle for the truncated process, for an arbitrary truncation parameter. Conclusion will follow by letting the truncation parameter go to infinity, using a control uniform on $\varepsilon$ on the probability that the initial and modified processes split before a given time $T>0$. This section is devoted to this point.

For an arbitrary $R\geqslant 1$ (considered as fixed until the very end of the proof of Theorem~\ref{thm:averaging}, where we will send it to infinity), introduce the function $g\in\mathcal C^2(\R^m,\R^m)$ given by
\begin{equation}
\label{loc:g}
g(\theta_\ell,\theta_b) = \po \theta_\ell, \chi(|\theta_b|/R) \theta_b \pf \,.
\end{equation}
for some $\mathcal C^2$ cut-off  function $\chi$, non-increasing, with $\chi(r)=1$ for $r\leqslant 1$ and $\chi(r) = 1/r $ for $r\geqslant 2$. Then $g$ is $1$-Lipschitz continuous and its second derivative is bounded by some constant $C_g>0$ uniformly over $R\geqslant 1$. Moreover, Proposition~\ref{prop:lambda_thetaPL} implies that 
\begin{equation}
\label{eq:lambdag}
\lambda := \inf_{\theta\in\R^m} \lambda_{g(\theta)} > 0\,.  
\end{equation}
This will enable to prove the averaging principle for the truncated process in the next section.

Since $\theta = g(\theta)$ as long as $|\theta_b| \leqslant R$, we will use the following to eventually let $R\rightarrow \infty$.

\begin{lemma}\label{lem:exit}
For any $T>0$, decomposing the solutions of~\eqref{eq:thetamucouples-epsilon} and~\eqref{eq:Theta0} as $(\Theta_{\ell,t}^\varepsilon,\Theta_{b,t}^\varepsilon) \in \R^{m_\ell}\times \R^{m_b}$ and $(\Theta_{\ell,t}^0,\Theta_{b,t}^0) \in \R^{m_\ell}\times \R^{m_b}$,
\[\sup_{\varepsilon\in(0,1]} \mathbb P\po \sup_{t\in[0,T]} |\Theta_{b,t}^\varepsilon| + |\Theta_{b,t}^0| \geqslant K \pf  \underset{K\rightarrow \infty}\longrightarrow 0 \,. \]
\end{lemma}

\begin{proof}
First,
\[\partial_t \mathbb E \po |\Theta_{b,t}^\varepsilon|^2 \pf \leqslant - \mathbb E \po |\Theta_{b,t}^\varepsilon|^2 \pf + \|\varphi_b\|_\infty^2 + 4m_b\beta^{-1}\,,\]
which ensures that
\begin{equation}\label{loc:dzsdfsdf}
\sup_{t\geqslant 0} \mathbb E \po |\Theta_{b,t}^\varepsilon|^2 \pf \leqslant M_b := \mathbb E \po |\Theta_{b,0}^\varepsilon|^2 \pf  + \|\varphi_b\|_\infty^2 + 4m_b\beta^{-1}\,.
\end{equation}
The same goes for $\Theta_{b,t}^0$. Next,
\[\dd |\Theta_{b,t}^\varepsilon|^2  \leqslant \po - |\Theta_{b,t}^\varepsilon|^2   + \|\varphi_b\|_\infty^2 + 4m_b\beta^{-1}\pf \dd t + 2\sqrt{2/\beta}|\Theta_{b,t}^\varepsilon|\dd B_t\,,\]
and thus, integrating in time,
\[|\Theta_{b,t}^\varepsilon|^2  \leqslant  t \po \|\varphi_b\|_\infty^2 + 4m_b\beta^{-1}\pf  + 2\sqrt{2/\beta}\int_0^t|\Theta_{b,s}^\varepsilon|\dd B_s\,.\]
The Burkholder-Davis-Gundy inequality and~\eqref{loc:dzsdfsdf} gives
\[\mathbb E \po \sup_{t\in[0,T]} |\Theta_{b,t}^\varepsilon| \pf \leqslant 8\beta^{-1} T M_b,\]
for any $\varepsilon\in(0,1]$, and similarly for $|\Theta_{b,t}^0|$. The Markov inequality concludes the proof.

\end{proof}

\subsubsection{Averaging for the truncated process}

In this section, for a fixed $R\geqslant 1$ and $g$ given in~\eqref{loc:g}, we consider the truncated processes
\begin{equation}
\label{eq:thetamucouples-epsilon-g}
\left\{
\begin{array}{rcl}
\partial_t \rho_t^\varepsilon &=&  \frac1\varepsilon \na\cdot \po \rho_t^\varepsilon\co  D\mathcal E_0(\rho_t^\varepsilon,\cdot) -  g\po \Psi_t^\varepsilon\pf \cdot \na \varphi\cf \pf  + \frac1\varepsilon \Delta\rho_t^\varepsilon  \\
\dd  \Psi_t^\varepsilon &=& -  \Psi_t^\varepsilon  + \rho_t^\varepsilon (\varphi)   + \sqrt{2/\beta} \dd B_t\,,
\end{array}\right.
\end{equation}
and
\begin{equation}
\label{eq:thetamucouples-0-g}
\dd  \Psi_t^0 = -  \Psi_t^0  + \mu_{g(\Psi_t^0)} (\varphi)   + \sqrt{2/\beta} \dd B_t\,,
\end{equation}
with initial conditions $\Psi_0^\varepsilon = \Psi_0^0 = \Theta_0$ and $\rho_0^\varepsilon = \nu_0$. The goal of this section is to establish the following:

\begin{proposition}
\label{prop:averaging}
Under the settings of Theorem~\ref{thm:averaging}, there exists $\varepsilon_0\in(0,1)$  such that for any $T>0$, there exists $C_T>0$ such that, for all $\varepsilon\in(0,\varepsilon_0]$, setting $ t_\varepsilon := \frac{2\varepsilon}{\lambda} \ln \varepsilon^{-1}$,
\begin{equation}
\label{eq:propaveraging}
\sup_{t\in[0,T]} \mathbb E \po |\Psi_t^\varepsilon- \Psi_t^0|^2  \pf + \sup_{t\in[t_\varepsilon,T]} \mathbb E \po  \mathcal W_2^2(\rho_t^\varepsilon,\mu_{g(\Psi_t^0)}) \pf  \leqslant C_T \varepsilon\ln \varepsilon^{-1}\,.
\end{equation}
\end{proposition}

\begin{remark}\label{rem:2}
In the case where $\varphi$ is linear (i.e. $m_b=0$) then $g(\theta)=\theta$ for all $\theta\in\R^m$, hence the processes~\eqref{eq:thetamucouples-epsilon-g} and~\eqref{eq:thetamucouples-0-g} are equal to the original ones~\eqref{eq:thetamucouples-epsilon} and~\eqref{eq:Theta0}. In this situation, Proposition~\ref{prop:averaging} is a stronger, more quantitative result than Theorem~\ref{thm:averaging}. 
\end{remark}

In order to prove Proposition~\ref{prop:averaging}, we start with some auxiliary results.

\begin{lemma}\label{lem:epsilon}
Under the settings of Theorem~\ref{thm:averaging}, there exists $\varepsilon_0>0$ such that, for all $T>0$,
\begin{equation}\label{eq:moment2}
\mathrm{M}_T := \sup_{t\in[0,T]} \sup_{\varepsilon\in(0,\varepsilon_0]} \mathbb E \po |\Psi_t^\varepsilon|^2 + \tF_{g(\Psi_t^\varepsilon)}(\rho_t^\varepsilon) + |\rho_t^\varepsilon(\varphi)|^2 \pf  < \infty \,.
\end{equation}
\end{lemma}
\begin{proof}
To alleviate notations we write $(\Psi,\rho)=(\Psi^\varepsilon,\rho^\varepsilon)$. Let $\lambda$ be given by~\eqref{eq:lambdag}, and recall that $g$ is $1$-Lipschitz continuous and $|\Delta g|\leqslant C_g$.  Applying It$\bar{\text{o}}$ formula to $(\nu,\theta) \mapsto  \mathcal F_{g(\theta)}(\nu)$,
\begin{align}
\partial_t \mathbb E \po \mathcal F_{g(\Psi_t)}(\rho_t)\pf &= \mathbb E \po - \frac{1}{\varepsilon} \mathcal I_{g(\Psi_t)}(\rho_t)  + \rho_t(\varphi)\cdot \na g(\Psi_t) \po   \rho_t(\varphi) - \Psi_t\pf + \rho_t(\varphi)\cdot \beta^{-1} \Delta g(\Psi_t) \pf  \nonumber\\
&\leqslant  \mathbb E \po - \frac{2\lambda}{\varepsilon} \tF_{g(\Psi_t)}(\rho_t)  + 3|\rho_t(\varphi)|^2 + |\Psi_t|^2 \pf + C_g^2 \beta^{-2} \,.\label{loc:Fthetaepsilon}
\end{align}
Besides, writing $L_\varphi$ the Lipschitz constant of $\varphi$,
\begin{align}
|\rho_t(\varphi)|^2 &\leqslant 3|\rho_t(\varphi)- \mu_{g(\Psi_t)}(\varphi)|^2 + 3 |\mu_{g(\Psi_t)}(\varphi)-\mu_0(\varphi)|^2 + 3|\mu_0(\varphi)|^2\nonumber\\
&\leqslant 3L_\varphi^2 \mathcal W_2^2 (\rho_t,\mu_{g(\Psi_t)}) + 3 L_{\varphi}^2 \mathcal W_2^2(\mu_{g(\Psi_t)},\mu_0) + 3|\mu_0(\varphi)|^2\nonumber\\
&\leqslant 6 L_\varphi^2 \lambda^{-1} \tF_{g(\Psi_t)}(\rho_t) + 27L_{\varphi}^4 \lambda^{-2} |\Psi_t|^2 + 3|\mu_0(\varphi)|^2\,,\label{loc:nutvarphi}
\end{align}
where we used the transport inequality~\eqref{eq:Talagrand}, Lemma~\ref{lem:muthetaLip} and that $|g(\psi)|\leqslant |\psi|$. Finally,
\begin{align*}
\partial_t \mathbb E \po |\Psi_t|^2 \pf &= \frac{4m}{\beta}+ 2 \mathbb E \po \Psi_t  (-  \Psi_t + \rho_t (\varphi))\pf \ \leqslant \  \frac{4m}{\beta}+  \mathbb E \po | \rho_t (\varphi)|^2 \pf\,.
\end{align*}
Combining the three previous inequalities, we get a constant $C_0>0$ independent from $\varepsilon$ such that 
\begin{align*}
\partial_t \mathbb E \po \mathcal F_{g(\Psi_t)}(\rho_t) +  |\Psi_t|^2  \pf
&\leqslant  C_0 + \mathbb E \po \po C_0- \frac{2\lambda}{\varepsilon}\pf \tF_{g(\Psi_t)}(\rho_t) + C_0 |\Psi_t|^2 \pf \,.
\end{align*}
Setting $\varepsilon_0 = 2\lambda/C_0$ we get, for any $\varepsilon\in(0,\varepsilon_0]$,
\[\partial_t \mathbb E \po \mathcal F_{g(\Psi_t)}(\rho_t) +  |\Psi_t|^2  \pf
\leqslant  C_0 +     C_0 \mathbb E \po  |\Psi_t|^2 \pf\,.\]
From~\eqref{loc:sfgkdflsdgz},
\[0 \leqslant \mathbb E \po \tF_{g(\Psi_t)}(\rho_t)\pf \leqslant \mathbb E \po \mathcal F_{g(\Psi_t)}(\rho_t)   + C_1(1+|\Psi_t|^2)\pf  \]
and thus we get, for some $C_2>0$,
\[0 \leqslant \mathbb E \po \tF_{g(\Psi_t)}(\rho_t) + |\Psi_t|^2 \pf \leqslant \mathbb E \po \mathcal F_{g(\Psi_0)}(\rho_0) + |\Psi_0|^2 \pf +  C_2 t + C_2 \int_0^t|\Psi_s|^2\dd s\]
 Grönwall lemma and~\eqref{loc:nutvarphi} concludes the proof.  
\end{proof}

In the rest of this section, $\varepsilon_0$ is as given by Lemma~\ref{lem:epsilon}.

\begin{lemma}\label{lem:Rs}
Under the settings of Theorem~\ref{thm:averaging}, for all $T>0$, there exists $\mathrm{K}_T >0$ such that for all $\varepsilon\in(0,\varepsilon_0]$, $t\in[0,T]$ and $s\in(0,1]$,
\[\mathbb E \po |\Psi_t^\varepsilon - \Psi_{t+s}^\varepsilon|^2 \pf \leqslant \mathrm{K}_T s\,.\]
\end{lemma}
\begin{proof}
This directly follows from Lemma~\ref{lem:epsilon} and 
\[\mathbb E \po |\Psi_t^\varepsilon - \Psi_{t+s}^\varepsilon|^2 \pf \leqslant 2 s \int_t^{t+s} \mathbb E \po |\Psi_u^\varepsilon|^2 + |\rho_u(\varphi)|^2 \pf \dd u + \frac{4}{\beta} \mathbb E \po |B_{t+s} - B_t|^2 \pf\,.
\]
\end{proof}

\begin{proposition}\label{prop:ftildeepsilon}
Under the settings of Theorem~\ref{thm:averaging}, for all $T>0$, there exists $\mathrm{K}_T'>0$ such that for all $\varepsilon \in(0,\varepsilon_0]$, setting $t_\varepsilon := \frac{2\varepsilon}{\lambda}\ln \varepsilon^{-1}$ (with $\lambda$ given by~\eqref{eq:lambdag}),
\[\sup_{t\in[t_\varepsilon,T]}\mathbb E \po \mathcal W_2^2(\mu_{g(\Psi_{t}^\varepsilon)},\rho_{t}^\varepsilon) \pf  \leqslant \mathrm{K}_T' \varepsilon \ln \varepsilon^{-1}\,. \]
\end{proposition}

\begin{proof}
Again, we simply write $(\Psi_{t}^\varepsilon,\rho_{t}^\varepsilon) =(\Psi_{t},\rho_{t}) $.
Proceeding as in the proof Lemma~\ref{lem:muthetaLip},
\begin{align*}
\partial_s \mathbb E \po \tF_{g(\Psi_t)}(\rho_{t+s}) \pf &= \partial_{|u=0} \mathbb E \po \tF_{g(\Psi_{t+s})}(\rho_{t+s+u}) + (\Psi_{t+s}-\Psi_t)\cdot \rho_{s+t+u}(\varphi) \pf \\
& = \mathbb E \po - \frac{1}{\varepsilon}\mathcal I _{g(\Psi_{t+s})}(\rho_{t+s}) + (\Psi_{t+s}-\Psi_t)\cdot \partial_s \rho_{s+t}(\varphi) \pf \\ 
& \leqslant   \mathbb E \po - \frac{1}{4\varepsilon}\mathcal I_{g(\Psi_{t+s})}(\rho_{t+s}) + \frac{2L_\varphi^2}{\varepsilon}  |\Psi_{t+s}-\Psi_t|^2 \pf \,. 
\end{align*}
Using the PL inequality~\eqref{eq:PLFtheta} and Lemma~\ref{lem:Rs} gives, for $t\in[0,T]$ and $s\in[0,1]$;
\[\partial_s \mathbb E \po \tF_{g(\Psi_t)}(\rho_{t+s}) \pf \leqslant - \frac{\lambda}{2\varepsilon}  \mathbb E \po \tF_{g(\Psi_t)}(\rho_{t+s}) \pf + \frac{2 \mathrm{K}_TL_\varphi^2}{\varepsilon}  s\,,\]
which we integrate over $s\in[0,\delta]$ for some $\delta\in(0,1]$ to get 
\[\mathbb E \po \tF_{g(\Psi_t)}(\rho_{t+\delta}) \pf \leqslant e^{-\frac{\delta \lambda}{2\varepsilon}} \mathbb E \po \tF_{g(\Psi_t)}(\rho_{t}) \pf + \frac{1 - e^{-\frac{\delta \lambda}{2\varepsilon}}}{\lambda/(2\varepsilon)}  \times \frac{2 \mathrm{K}_T L_\varphi^2}{\varepsilon}  \delta \leqslant \mathrm{M}_T e^{-\frac{\delta \lambda}{2\varepsilon}} + \frac{4 \mathrm{K}_T L_\varphi^2}\lambda  \delta\,.   \]
Then, applying Lemma~\ref{lem:Ftildepsitheta},
 \begin{align*}
 \mathbb E \po \tF_{g(\Psi_{t+\delta})}(\rho_{t+\delta}) \pf & \leqslant \mathbb E \po 2 \tF_{g(\Psi_t)}(\rho_{t+\delta}) + \frac{4L_\varphi^2}{\lambda}|\Psi_{t+\delta} - \Psi_t|^2 \pf  \\
 & \leqslant   2 \mathrm{M}_T e^{-\frac{\delta \lambda}{2\varepsilon}}  + \frac{12 L_\varphi^2}\lambda   \mathrm{K}_T \delta\,.
 \end{align*}
 Applying this for $\delta = t_\varepsilon$ and using~\eqref{eq:Talagrand} concludes the proof.
\end{proof}

 \begin{proof}[Proof of Proposition~\ref{prop:averaging}]
Since the processes are driven by the same Brownian motion $B$,
\begin{align*}
\Psi_t^\varepsilon - \Psi_t^0 &= \int_0^t \po \Psi_s^0 - \Psi_s^\varepsilon + \rho_{s}^\varepsilon(\varphi) - \mu_{g(\Psi_s^0)}(\varphi)  \pf  \dd s \,.
\end{align*}
Using Lemma~\ref{lem:muthetaLip} and that $g$ is $1$-Lipschitz continuous,
\begin{align*}
|\Psi_t^\varepsilon - \Psi_t^0| & \leqslant  \int_0^t \po (1+3 L_\varphi\lambda^{-1})|\Psi_s^0 - \Psi_s^\varepsilon| + |\rho_{s}^\varepsilon(\varphi) - \mu_{g(\Psi_s^\varepsilon)}(\varphi)|  \pf  \dd s  \\
& \leqslant    \int_0^t \po (1+3 L_\varphi\lambda^{-1})|\Psi_s^0 - \Psi_s^\varepsilon| + L_\varphi\mathcal W_2(\rho_{s}^\varepsilon,\mu_{g(\Psi_s^\varepsilon)})  \pf  \dd s\,.
\end{align*}
Taking the square, the expectation and applying Grönwall's lemma, for any $T>0$, there exists $C_T>0$ such that
\[\sup_{t\in[0,T]} \mathbb E \po |\Psi_t^\varepsilon - \Psi_t^0|^2\pf \leqslant C_T \int_0^T \mathbb E \po \mathcal W_2^2(\rho_{s}^\varepsilon,\mu_{g(\Psi_s^\varepsilon)})  \pf  \dd s \,.\]
 Thanks to Lemma~\ref{lem:epsilon} (with~\eqref{eq:Talagrand}) and Proposition~\ref{prop:ftildeepsilon}, distinguishing whether $t\gtrless t_\varepsilon$, 
\[\int_0^T \mathbb E \po \mathcal W_2^2(\rho_{s}^\varepsilon,\mu_{g(\Psi_s^\varepsilon)})  \pf   \dd s \leqslant 2 \lambda^{-1} \mathrm{M}_T t_\varepsilon +  T \mathrm{K}'_T\varepsilon \ln \varepsilon^{-1} \,,\]
which concludes the proof of the bound for the first term in the left-hand side of~\eqref{eq:thmaveraging}. For the second term, it remains to bound, for $t\geqslant t_\varepsilon$,
\[\mathbb E \po \mathcal W_2^2(\rho_{t}^\varepsilon,\mu_{g(\Psi_t^0)}) \pf \leqslant 2\mathbb E \po \mathcal W_2^2(\rho_{t}^\varepsilon,\mu_{g(\Psi_t^\varepsilon)}) + 9L_\varphi^2\lambda^{-2} |\Psi_t^\varepsilon - \Psi_t^0|^2 \pf\,.  \]
The bound already proven for the first term of~\eqref{eq:thmaveraging} and Proposition~\ref{prop:ftildeepsilon} concludes the proof of Proposition~\ref{prop:averaging}.
\end{proof}

\subsubsection{Conclusion of the proof}

\begin{proof}[Proof of Theorem~\ref{thm:averaging}]
Fix $T>t_0>0$.  For an arbitrary $R\geqslant 1$, consider the event
\[\mathcal A_R = \left\{\sup_{t\in [0,T]}|\Theta_{b,t}^\varepsilon|+|\Theta_{b,t}^0| \leqslant R\right\}\,,  \]
and let $g$ be given by~\eqref{loc:g} and $\Psi^\varepsilon,\rho^\varepsilon,\Psi^0$ as in~\eqref{eq:thetamucouples-epsilon-g} and~\eqref{eq:thetamucouples-0-g}. Under the event $\mathcal A_R$, for all $t\in[0,T]$, $(\Psi_t^\varepsilon,\rho_t^\varepsilon,\Psi_t^0)=( \Theta_t^\varepsilon,\nu_t^\varepsilon,\Theta_t^0)$. For $t\in[0,T]$,  we bound
\begin{align*}
\mathbb P \po |\Theta_t^\varepsilon- \Theta_t^0| \geqslant \delta  \pf &  \leqslant \mathbb P \po |\Psi_t^\varepsilon- \Psi_t^0| \geqslant \delta  \pf + \mathbb P(\mathcal A_R^c)\,.
\end{align*}
By the Markov inequality and Proposition~\ref{prop:averaging}, 
\[\limsup_{\varepsilon\rightarrow 0}\sup_{t\in[0,T]} \mathbb P \po |\Theta_t^\varepsilon- \Theta_t^0| \geqslant \delta  \pf \leqslant \mathbb P(\mathcal A_R^c)\,. \]
Using that, for any $t_0>0$, for any $R\geqslant 1$, $t_\varepsilon \leqslant t_0$ for $\varepsilon$ small enough (with $t_\varepsilon$ as in Proposition~\ref{prop:averaging}; it depends on $R$ through $\lambda$), we get with the same argument that
\[\limsup_{\varepsilon\rightarrow 0}\sup_{t\in[t_0,T]} \mathbb P \po \mathcal W_2(\nu_t^\varepsilon,\mu_{\Theta_t^0}) \geqslant \delta  \pf \leqslant \mathbb P(\mathcal A_R^c)\,. \]
Letting $R\rightarrow \infty$ thanks to Lemma~\ref{lem:exit} concludes the proof.
\end{proof}

\subsection*{Acknowledgments}

The research of P. Monmarché is supported by the project CONVIVIALITY (ANR-23-
CE40-0003) of the French National Research Agency. Generative AI tools have been used only for English language in Section~\ref{sec:numerique}.

\bibliographystyle{plain}
\bibliography{biblio}

\end{document}